\documentclass[a4paper,11pt,reqno]{amsart}
\usepackage{amsmath}
\usepackage{amsthm,enumerate}
\usepackage{mathtools}
\usepackage{enumitem}

\usepackage{graphicx}
\usepackage{amssymb}

\usepackage{tikz}
\usepackage{tikz-qtree}
\usepackage{comment}
\usetikzlibrary{calc}
\usetikzlibrary{patterns}
\usetikzlibrary{patterns.meta}
\usetikzlibrary{trees}
\usepackage[toc,page]{appendix}

\usepackage[utf8]{inputenc}

\usepackage{color}

\usepackage{url}

\usepackage[text={6.5in,8.6in},centering]{geometry}

\usepackage{srcltx}

\usepackage[T1]{fontenc}

\theoremstyle{plain}
\newtheorem{thm}{Theorem}[section]
\theoremstyle{plain}
\newtheorem{lem}[thm]{Lemma}
\newtheorem{prop}[thm]{Proposition}

\theoremstyle{definition}
\newtheorem{defi}[thm]{Definition}
\newtheorem{rem}[thm]{Remark}

\newcommand{\R}{\mathbb{R}}

\newcommand{\hn}{V}

\newcommand{\GM}{\mathbb{G}}
\newcommand{\N}{\mathbb{N}}

\newcommand{\dg}{\deg}

\numberwithin{equation}{section} \allowdisplaybreaks

\begin{document}
 \title[Fractional Porous Medium Equation on graphs]{The fractional Porous Medium Equation\\ on   graphs}

\author[Berchio]{Elvise Berchio}
\address{Dipartimento di Scienze Matematiche ``Giuseppe Luigi Lagrange'',
  Politecnico di Torino, Corso Duca degli Abruzzi 24, 10129 Torino,
  Italy}
\email{elvise.berchio@polito.it}

\author[Santagati]{Federico Santagati}
\address{Dipartimento di Scienze Matematiche ``Giuseppe Luigi Lagrange'',
  Politecnico di Torino, Corso Duca degli Abruzzi 24, 10129 Torino,
  Italy}
\email{federico.santagati@polito.it}

\author[Vallarino]{Maria Vallarino}
\address{Dipartimento di Scienze Matematiche ``Giuseppe Luigi Lagrange'',
  Politecnico di Torino, Corso Duca degli Abruzzi 24, 10129 Torino,
  Italy}
\email{maria.vallarino@polito.it}

\date{}




\begin{abstract}
We study the fractional porous medium equation on connected infinite graphs with no local finiteness assumption.
We introduce a notion of weak dual solution adapted to the discrete setting, and establish existence results for nonnegative initial data belonging to a weighted space defined through the fractional Green function, extending beyond the classical $\ell^1$ framework. Our approach relies on weighted estimates and on a detailed analysis of the associated fractional Green function. In the particular case of infinite trees with standard weights, we establish comparison principles and derive estimates for the fractional Green function, which lead to quantitative smoothing effects for solutions. 
\end{abstract}

\maketitle


\normalsize


 \section{Introduction}
The study of nonlinear diffusion equations on discrete settings has attracted increasing attention in recent years, motivated both by applications and by the deep interplay between analysis, probability, and geometry on graphs,   {see e.g., \cite{BBSV,GMeP,Hu-Wang,Hua-Mugnolo,Lenz-Schmidt-Zimmermann,Lin-Wu2,Lin-Yang,MPS,Punzo-Sacco} and references therein}.  Given a connected infinite graph $\Gamma=(V,b)$, this paper is devoted to the study of existence and smoothing estimates of solutions to the fractional porous medium equation on $\Gamma$. More precisely, we consider nonnegative solutions to the Cauchy problem:
\begin{equation}\label{NFDE}
\left \{ \begin{array}{ll}
\partial_t u (t,x)+ (- \Delta)^s u^m(t,x)= 0\,, &  (t,x)\in  (0, \infty) \times V,\\
u(0,x)=u_0(x)\ge0\,, & x\in V,
\end{array}
\right.
\end{equation}
where $0<s\leq 1$ and $m>1$. See Subsections \ref{1} and \ref{2} for the precise definition of the graph setting and of the operator.

When $s=1$, the equation reduces to the classical porous medium equation on graphs. In the Euclidean space, porous medium type equations have been extensively studied and a comprehensive account of the theory can be found in the monograph \cite{V}. The analysis includes existence and uniqueness of solutions, regularizing effects, asymptotic behavior and qualitative properties. More recently, strong attention has been devoted to nonlocal nonlinear diffusion equations, see for instance  \cite{BFR, BFV,BSV, BV1,Vjems}, where smoothing estimates, sharp quantitative bounds and propagation properties have been established.

On Riemannian manifolds, the geometry strongly influences the behavior of nonlinear diffusion equations. In particular, curvature assumptions and volume growth conditions play a crucial role in determining smoothing effects, asymptotic profiles and admissible classes of initial data. We refer to \cite{GMP2,GMP,GMV,GMV-MA,V4} and the references therein for some representative contributions. Fractional porous medium equations on manifolds have been recently studied in \cite{BBGG,BBGM}, where the analysis combines nonlocal methods with geometric estimates.

To the best of our knowledge, the fractional porous medium equation on infinite graphs has not yet been investigated. In the local case $s=1$,  existence and uniqueness results for the porous medium equation with $\ell^1$ initial data on graphs were proved in \cite{BSW}, see also \cite{Ma} where a logarithmic diffusion equation of porous-medium type is considered. Aronson-Bénilan and Harnack estimates were obtained in \cite{KZ}, and nonexistence results for nonlinear diffusion equations with reaction terms were established in \cite{vC}.

The first aim of this paper is to determine proper assumptions so that nonnegative \emph{weak dual solutions} (WDS) to \eqref{NFDE} exist. This notion was originally introduced in the Euclidean setting in \cite{BV2} and subsequently developed for several classes of nonlinear nonlocal equations on bounded domains, see e.g.,  \cite{BFR,BFV,BSV, BV3, BV1,Vjems}. In particular, the WDS definition is based on a reformulation of the equation through the inverse fractional operator and the corresponding existence theory is usually developed by combining nonlinear semigroup techniques with potential-theoretic methods based on the Green function associated with the operator. A general framework where this approach works, including smoothing effects and functional inequalities, is provided in \cite{BE23}. On manifolds, related techniques have been employed in \cite{BBGG,BBGM,GMP0}.

As far as we know, WDS have never been considered in the framework of graphs. Although our approach is inspired by the theory developed for nonlocal nonlinear diffusion equations in continuous settings, its extension to graphs requires substantial modifications and cannot be regarded as a straightforward discrete counterpart of the existing theory. Indeed, several structural features of the graph setting differ significantly from their continuous analogues. First, unlike several existing works on manifolds and graphs, our analysis does not require bounded geometry assumptions, volume doubling properties, or any control on the growth of metric balls. The only structural assumption is a Nash inequality which yields heat kernel estimates on which our analysis is based. In particular, no spectral gap is required for the Laplacian. Second, under our assumptions the fractional Green function turns out to be bounded, in sharp contrast with the singular behavior typically exhibited by Green kernels, e.g, in Euclidean spaces. Moreover, the equation is naturally interpreted pointwise on the vertices rather than in a distributional sense, and the very definition of WDS has to be adapted accordingly. Finally, the scale of spaces $\ell^p(V)$ behaves differently from the $L^p$-scale in continuous settings.

 We also point out that besides providing a natural extension of the continuous theory, the introduction of the WDS notion is essential for our purposes. Indeed, it allows us to exploit directly the fractional Green function associated with the graph Laplacian and to derive weighted $\ell^1$-estimates that are not available at the level of mild solutions. Such estimates constitute the key tool for constructing, in Subsection \ref{subsectionexistence} , WDS corresponding to initial data belonging to a weighted Green-function space which strictly contain $\ell^1(V)$ and they provide the starting point for the smoothing effects proved in Section \ref{trees}.

The second main goal of the paper is to investigate the regularizing effects enjoyed by WDS on trees, i.e. infinite graphs without loops, with standard weights,  satisfying a uniform lower bound on the number of neighbors. In particular, a substantial part of the paper is devoted to establishing new integral estimates for the heat kernel and the fractional Green function on trees through comparison arguments with homogeneous trees. Besides being instrumental for our analysis, these estimates are of independent interest, and may be useful in the study of other nonlinear and nonlocal equations. Exploiting the obtained estimates for the fractional Green function we prove quantitative smoothing estimates of WDS of \eqref{NFDE} corresponding to data in $\ell^1(V)$ and also for initial data belonging to weighted Green-function spaces (see Theorems \ref{thm.smoothing.HN-like} and \ref{thm.smoothing.pesato} below).

%
%
%
The paper is organized as follows. In Section \ref{main} we introduce the graph framework, the fractional Laplacian and the associated Green function, and we state key estimates of the Green function on graphs. Section \ref{WDS_sect} is devoted to the definition of WDS, to the proof that mild solutions with initial data in $\ell^1(V)$ are WDS and some related a-priori estimates. Moreover, in Theorem \ref{existenceWDS} we prove existence of WDS for initial data in the above mentioned weighted space defined through the Green function. Section \ref{trees} contains the smoothing effects on trees for initial data in $\ell^1$ and on the weighted space. The final section is devoted to the proof of several auxiliary results on trees, including comparison principles, {monotonicity results} and integral heat kernel estimates, which are used in the {proofs of the integral estimates} of the fractional Green function. 

Throughout the paper, $C$ denotes a positive constant that may vary from line to line. Dependence on relevant parameters will be indicated when needed.

 \section{Preliminaries and Green functions estimates}\label{main}
\subsection{Graph setting}\label{1}
We consider a {graph} $\Gamma=(V,b)$ where $V$ is a countable set of nodes and $b\colon V\times V \to [0,\infty)$ is a weight function satisfying the following properties:
\begin{enumerate}
	\item $b_{xy}=b_{yx}$ for every $x,y \in V$;
	\item $b_{xx}=0$ for every $x \in V$;
	\item $\deg(x)=\sum_{y\in V} b_{xy} < \infty$ for every $x \in V$, and $\deg(x)$ is called the degree of the node $x$.
\end{enumerate}
Vertices $x,y\in V$ are said to be neighbors if $b_{xy}>0$ and we write $x\sim y$ in this case. Notice that we do not assume that $\mathrm{deg}$ is a bounded function on $V$. Moreover, in general, (3) does not imply that each vertex has finitely many neighbors.

We denote by $d$ the standard discrete distance on $V$, by $C(V)$ the set of all real valued functions on $V$, and by $C_c(V)$ the subset of compactly supported (i.e., finitely supported) functions in $C(V)$. Throughout the paper, \( C(V) \) is endowed with the topology of pointwise convergence, i.e., the Fréchet topology generated by the seminorms \( p_x(f) = |f(x)| \), \( x \in V \). For $1\leq p<\infty$ the standard $\ell^p(V) $ space is defined as 
$$
\ell^p(V)=\Big\{u\in C(V):\Big( \sum_{x\in V}|u(x)|^p\deg(x)\Big)^{1/p}=\|u\|_{\ell ^p(V)}<\infty\Big\},
$$
and we denote by $\ell^{\infty}(V)$ the set of real functions $u$ on $V$ such that $\|u\|_{\infty}=\sup_{x\in V}|u(x)|<\infty$.

We shall also assume the following condition:
  there exists a number $n>2$ such that the Nash inequality holds
    \begin{align}\label{Nash}
        \|u\|_{\ell^2(V)}^{2+4/n}
        \le
        C\,\mathcal E (u,u)\,
        \|u\|_{\ell^1(V)}^{4/n} \qquad \forall u\in C_c(V),
    \end{align}
 where $$\mathcal{E}(u,u)=\frac{1}{2}\sum_{x,y\in V}b_{xy}(u(y)-u(x))^2.$$

   \subsection{The Laplacian, the fractional Laplacian and its left-inverse}\label{2}
We define the Laplacian on $\Gamma = (V,b)$ by
\begin{align*}
(-\Delta u)(x) = u(x) - \frac{1}{\deg(x)} \sum_{y \in V} b_{xy}\,u(y),
\qquad x \in V,
\end{align*}
for every $u$ in the domain
\[
\mathrm{dom}(-\Delta)
= \left\{ f \in C(V) \;:\; \sum_{y \in V} b_{xy}\,|f(y)| < \infty \ \text{for every } x \in V \right\}.
\]
In particular, all spaces $\ell^p(V)$, with $1\leq p\le \infty$, are included in $\mathrm{dom}(-\Delta)$. 

Notice that $-\Delta =I-P$, where $P$ is the operator whose kernel with respect to $\deg$ is 
$$
p(x,y)=
\frac{ b_{xy}}{\deg(x)\deg(y)}\qquad \forall x,y\in V,$$
and $$\sum_{y \in V} p(x,y) \, \deg(y)=1\qquad \forall x\in V.$$
Hence, $-\Delta$ is bounded on $\ell^p(V)$, for every $1\leq p\leq \infty$, and generates a strongly continuous contraction semigroup $e^{t\Delta}$ on $\ell^1(V)$ which is Markovian. We point out that, by definition of the Laplacian, inequality \eqref{Nash} can be reformulated as 
$$
  \|u\|_{\ell^2(V)}^{2+4/n}
        \le
        C\, \langle -\Delta u,u\rangle_{\ell^2(V)}\,
        \|u\|_{\ell^1(V)}^{4/n} \qquad \forall u\in C_c(V).
$$

For $0<s<1$ the fractional Laplacian is defined as follows:
$$
 (-\Delta)^s u(x) = \frac1{\Gamma(-s)}\int_0^{+\infty}   \left( \sum_{y \in V} h_t(x,y)\left(u(y)-u(x)\right)  \deg(y) \right)\, \frac{\rm{d}t}{t^{1+s}},
$$
where $\Gamma(-s)$ denotes the classical Gamma function while $ h_t$ denotes the   {kernel of the heat semigroup  $e^{t\Delta}$ with respect to the measure $\deg$}.  {The operator $(-\Delta)^s$} is bounded on $\ell^p(V)$, for every $1\leq p\leq \infty$, and by Bochner subordination, it  generates a strongly continuous contraction semigroup on $\ell^1(V)$ which is still sub-Markovian and in particular order-preserving,  see \cite[Corollary 4.3.4]{J}. For related semigroup-based definitions of fractional Laplace operators on locally finite graphs, we refer to \cite{ZLY}.

Then, for every $s \in (0,1]$, we define the fractional Green function as 
\begin{equation*}
\GM^s(x,y)= \frac1{\Gamma(s)}\int_0^{+\infty}\frac{h_t(x,y)}{t^{1-s}}\,{\rm{d}t}, 
\end{equation*}
and
\begin{equation*}
(-\Delta)^{-s} u (x)  =  \sum_{y \in V} u (y) \, \GM^s(x,y)\, \deg(y)  \,.
\end{equation*}
 We sometimes write $\GM$ in place of $\GM^1$.
By \cite[Theorem~3.25]{CarlenKusuokaStroock87} (see also \cite{Rose2025}),
the Nash inequality \eqref{Nash} implies the on-diagonal estimate
\begin{align}\label{ultracontr1}
    \sup_{x,y\in V} h_t(x,y) \le C\, t^{-n/2}\qquad \forall t>0,
\end{align}
or, equivalently, the ultracontractive bound
\begin{align}\label{ultracontr}
    \|e^{t\Delta}\|_{\ell^1(V)\to \ell^\infty(V)}
    \le
    C\, t^{-n/2}\qquad \forall t>0.
\end{align}

Another consequence of \eqref{Nash} is that the measure $\deg$ is uniformly non-degenerate, namely there exists a constant $d_0>0$ such that
\begin{align}\label{deg}
\deg(x)\ge d_0
\qquad \forall x\in V.
\end{align}

To see this, it is enough to test \eqref{Nash} with $u=\delta_x$, the Dirac delta at $x$. Indeed,
\begin{align*}
\deg(x)^{1+2/n}
=\|\delta_x\|_{\ell^2(V)}^{2+4/n}
\le C\,\mathcal E(\delta_x,\delta_x)\,
\|\delta_x\|_{\ell^1(V)}^{4/n}
=C\,\deg(x)^{1+4/n},
\end{align*}
which yields \eqref{deg} with $d_0=C^{-n/2}$. In turn, by \eqref{deg} it follows that $\|u\|_{\ell^{\infty}(V)}\leq d_0^{-\frac{1}{p}}\|u\|_{\ell^p(V)}$ and $\|u\|_{\ell^q(V)}\leq d_0^{\frac{p-q}{pq}}\|u\|_{\ell^p(V)}$ for every $1\leq p<q< \infty$ and $u\in \ell^p(V)$. Moreover, by \cite[Theorem~10, Corollary~11]{Davies93}, \eqref{deg} also implies the uniform estimate
\begin{align}\label{ultraConst}
    \sup_{x,y\in V} h_t(x,y) \le C \qquad \forall t>0.
\end{align}
Combining \eqref{ultracontr} and \eqref{ultraConst}, we obtain
\begin{align}\label{stimaheat}
    \|e^{t\Delta}\|_{\ell^1(V)\to \ell^\infty(V)}
    \le
    C\,\min\{1,t^{-n/2}\}\qquad \forall t>0.
\end{align}

\begin{rem}\label{rembottom}
It is worth noting that assuming \eqref{deg}, the Nash inequality \eqref{Nash} holds whenever $-\Delta$ has positive spectral gap, namely,
    \begin{equation}\label{bottompositivo}
        \inf\{z\in\sigma_2(-\Delta)\}=b_2>0.
    \end{equation}
    Indeed, in this case
    \[
        \langle- \Delta u,u\rangle_{\ell^2(V)}
        \ge
        b_2\,\|u\|_{\ell^2(V)}^2,
    \]
    and \eqref{Nash} (with $C=d_0^{2/n} b_2$ for any $n>0$) follows by recalling that \eqref{deg} implies that $\sqrt{d_0} \|u\|_{\ell^2(V)}\leq \|u\|_{\ell^1(V)}$. On the other hand, there are examples in which \eqref{Nash} holds while
    \[
        \inf\{z\in\sigma_2(-\Delta)\}=0.
    \]
    A typical example is the lattice $\mathbb{Z}^n$ with $n\ge3$
    (see \cite[Proposition~5.1]{GrigoryanTelcs01} and
    \cite[Chapter~1, Eq.~(1.4)]{Woess00}).

    
 \end{rem}

\subsection{Green functions estimates}\label{Greentrees}

An important consequence of \eqref{stimaheat} is the uniform boundedness of the fractional Green function:
    \[
        \sup_{x,y\in V} \GM^s(x,y)\leq  C_{s,n}= \frac{C}{\Gamma(s)}
\left(\frac1s+\frac1{\frac n2-s}\right)
        \qquad\text{for every } s\in(0,1]\,.
    \]

To obtain a useful pointwise estimate of the fractional Green function, we first notice that, since $e^{t\Delta}
=
e^{-t}e^{tP}
$  {and $P$ is bounded on every $\ell^p(V)$,} expanding the exponential
we get that
\[
h_t(x,y)
=
e^{-t}
\sum_{n=0}^\infty
\frac{t^n}{n!}
\,p_n(x,y),
\]
where $p_n(x,y)=(P^n)(x,y)$ is the $n$–step transition probability of the random walk with respect to $\deg$. Substituting this expression into the integral representation of $\GM^s$ gives
\[
\GM^s(x,y)
=
\frac{1}{\Gamma(s)}
\int_0^\infty
t^{s-1}e^{-t}
\sum_{n=0}^\infty
\frac{t^n}{n!}p_n(x,y)\,dt.
\]
Since all terms are nonnegative we may interchange sum and integral
and using the identity $\int_0^\infty t^{n+s-1}e^{-t}\,dt=\Gamma(n+s)$,
we obtain the alternative representation
\begin{align}\label{formulaGS}
\GM^s(x,y)
=
\sum_{n=0}^\infty
\frac{\Gamma(n+s)}{\Gamma(s)\,n!}\;
p_n(x,y).
\end{align}
By Stirling's formula,
$\frac{\Gamma(n+s)}{n!}
\sim
n^{\,s-1}$ as $n\to\infty$,
and the fact that $p_n(x,y)=0$ if $n<d(x,y)$, we deduce that
\begin{equation}\label{stimapuntualeGs}
\begin{aligned}
    \GM^s(x,y)&\leq C\sum_{n=d(x,y)}^\infty (n+1)^{s-1}p_n(x,y) \\&\le  \frac{C}{(d(x,y)+1)^{1-s}}  \sum_{n=d(x,y)}^\infty p_n(x,y)\\&= \frac{C}{(d(x,y)+1)^{1-s}} \, \GM(x,y) \qquad \forall x,y \in V\,.
\end{aligned}
\end{equation}

\smallskip

We can obtain sharper estimates for the fractional Green function in the setting of trees with standard weights, as a consequence of a comparison result between general and homogeneous trees, whose proof is postponed to Section~\ref{last}.


Let $T$ be a tree with set of vertices $V$  {and standard weights}, i.e., an infinite graph without loops, such that $b_{xy}\in  {\{0,1\}}$. We assume that there exists an integer $q\geq 2$ such that  {\begin{align}\label{degq+1}\mathrm{deg}(x) \ge q+1   \qquad x \in V.\end{align}} 
By \cite[Theorem 4.3]{Woj}, condition \eqref{degq+1} implies 
$$
b_2=\inf\{z\in\sigma_2(-\Delta)\}\geq \frac{(q-1)^2}{2(q+1)^2}>0.
$$
Hence by Remark \ref{rembottom} the Nash inequality \eqref{Nash} holds.

We denote by $T_q$ the homogeneous tree of degree $q+1$, i.e. the tree  {with standard weights} where $\deg$ is uniformly equal to $q+1$, and by $V_q$ its set of vertices. We shall denote by $-\Delta$ and
$-\Delta^{(q)}$ the Laplacians on $C(V)$ and $C(V_q)$, respectively, by $h_t$ and $h_t^{(q)}$
the corresponding heat kernels, and by $\GM^s$ and $\GM_q^s$ the corresponding fractional Green
functions. Finally, $d$ and $d_q$ stand for the discrete distances on $V$ and $V_q$, respectively.

The following comparison result holds. 


\begin{prop} \label{intestimateGreentrees}
 {Let $T$ be a tree with standard weights such that \eqref{degq+1} holds.} Then, for every $x,o\in V$, $r \in \mathbb N$ and $t>0$ 
\begin{align}\label{comp:heat}
    \sum_{y \in B_r(o)} h_t(x,y) \mathrm{deg}(y) \le (q+1)\sum_{y_q \in B_r(o_q)} h_t^{(q)}(x_q,y_q),
\end{align} where $x_q, o_q \in V_q$  are such that $d_q(x_q,o_q)=d(x,o)$.
In particular,
\[
h_t(o,x)\le \frac{q+1}{\mathrm{deg}(x)} h_t^{(q)}(o_q,x_q) \le h_t^{(q)}(o_q,x_q),
\] and 
\begin{align}\label{comp:green}
    \sum_{y \in B_r(o)} \GM^s(x,y) \mathrm{deg}(y) \le  (q+1)\sum_{y_q \in B_r(o_q)} \GM^s_q(x_q,y_q). 
    \end{align}  
\end{prop}
By \eqref{stimapuntualeGs} and Proposition \ref{intestimateGreentrees} applied with $r=0$, for all $x,y \in V$ we get 
\begin{equation*}
 \GM^s(x,y)\leq   \frac{C}{(d(x,y)+1)^{1-s}}\, \GM(x,y) \leq \frac{C}{(d(x,y)+1)^{1-s}}\, \GM_q(x_q,y_q), 
\end{equation*}
for some $x_q,y_q\in V_q$ such that  $d_q(x_q,y_q)=d(x,y)$. Finally, since   {$\GM_q(x_q,y_q)= \frac{q}{q-1}q^{-d_q(x_q,y_q)}$} (see e.g, \cite[Lemma 1.24]{Woess00}), we conclude that
\begin{equation}\label{stimaGs} 
 \GM^s(x,y)\leq    \frac{C}{(d(x,y)+1)^{1-s}}\, q^{-d(x,y)} \qquad \forall \, x,y\in V\,.
\end{equation}

 \section{Weak dual solutions (WDS)}\label{WDS_sect}

   {The aim of this section is to introduce the notion of WDS in the setting of graphs and to show that mild solutions belong to this class. This notion offers several advantages. In particular, it allows one to establish weighted $\ell^1$ a priori estimates involving the Green function, as well as existence results for solutions corresponding to initial data in a weighted space that strictly contains $\ell^1(V)$ (Remark~\ref{strict_incl}). These results are obtained on general graphs in Proposition~\ref{crucial} and Theorem~\ref{existenceWDS}. Moreover, in the special case of trees, where suitable estimates for the Green function are available, the weak dual formulation yields smoothing estimates (Section~\ref{trees}).
}

  \subsection{Definition of WDS}
By suitably adapting the definition given in \cite{BBGM} to the discrete setting, we consider the following {weighted space}, defined in terms of the fractional Green function:
 {\begin{equation*}\label{LG}
\ell^1_{\GM^s}(\hn) = \left \{ u \in C(V)  : \ \left\| u \right\|_{\ell^1_{\GM^s}(V)}  = \sup_{x_0\in V}   \left\| u \right\|_{\ell^1_{x_0,\GM^s}(V)}=\sup_{x_0\in V}\\\sum_{x \in V}  \left| u(x) \right| \GM^s(x ,x_0) \, \dg(x) \,  < +\infty  \right\}.
\end{equation*}}
Since $ \GM^s(x ,x_0)\in \ell^{\infty}(V)$ for all $ x_0 \in V $, it follows that $ \ell^1(V)  \subseteq  \ell^1_{\GM^s}(\hn)$. {As shown in Remark \ref{strict_incl}, this inclusion is strict whenever the bottom of the $\ell^2$-spectrum of $-\Delta$ is positive.  
 {\begin{rem}\label{strict_incl} Assume that $\Gamma=(V,b)$ is a graph with positive spectral gap, i.e., such that \eqref{bottompositivo} holds. Then  $\ell^1(V)$ is strictly contained in $\ell^1_{\GM^s}(V)$ for every $s\in(0,1]$.
Indeed, fix $x_0 \in V$. By \eqref{formulaGS}
\begin{align*}
    \GM^s(x_0,y) \ge C_s \sum_{n \ge 0} (1+n)^{s-1}p_n(x_0,y)   \qquad \forall y \in V.
\end{align*}
Since for every $n \ge 0$, $\displaystyle{\sum_{y \in V} p_n(x_0,y) \mathrm{deg}(y)=1},$ we deduce that 
\begin{align*}
    \sum_{y \in V} \GM^s(x_0,y) \mathrm{deg}(y)\ \ge C_s \sum_{n \ge 0} (1+n)^{s-1}\sum_{y \in V} p_n(x_0,y) \mathrm{deg}(y)=C_s\sum_{n \ge 0} (1+n)^{s-1}=+\infty,
\end{align*} since $s \in (0,1]$. This proves that $\GM^s(x_0,\cdot) \not\in \ell^1(V).$
Next, since $-\Delta$ has positive spectral gap, the operator  $(-\Delta)^{-s}$ is bounded on $\ell^2(V)$ and
\begin{align*}
    \sum_{y \in V} (\GM^s(y,x_0))^2 \mathrm{deg}(y)= \left\|(-\Delta)^{-s}\left(\frac{\delta_{x_0}}{\mathrm{deg}(x_0)}\right)\right\|_{\ell^2(V)}^2 \le c_{s} \frac{1}{\mathrm{deg}(x_0)},
\end{align*} for some $c_s>0$.
Since $\GM^s(y,x_0)=\GM^s(x_0,y)$, this proves that $\GM^s(x_0,\cdot) \in \ell^1_{x_0,\GM^s}(V)$. Finally, we show that $\GM^s(x_0,\cdot) \in \ell^1_{\GM^s}(V)$. Indeed, by \eqref{deg}, 

\begin{align*}
\|\GM^s(x_0,\cdot)\|_{\ell^1_{\GM^s}(V)}&
= \sup_{z \in V} \sum_{y \in V} \GM^s(x_0,y)\GM^s(z,y) \mathrm{deg}(y) \\ &\le \sup_{z \in V} \| \GM^s(x_0,\cdot)\|_{\ell^2(V)}\|\GM^s(z,\cdot)\|_{\ell^2(V)} \le \sup_{z \in V} \frac{c_s}{\sqrt{\mathrm{deg}(x_0)\mathrm{deg}(z)}}\le \frac{c_s}{d_0}.
\end{align*}
Hence, $\GM^s(x_0,\cdot) \in \ell^1_{\GM^s}(V) \setminus \ell^1(V).$
\end{rem}}

\begin{defi}\label{defi_WDS}
	
 Let $ u_0 \in \ell^1_{\GM^s}(V)$ with $ u_0 \ge 0 $. We say that a nonnegative measurable function $ u $    {on $[0,\infty) \times V$} is a Weak Dual Solution (WDS) to Problem~\eqref{NFDE} if, for every $T>0$:

\begin{itemize}
	
\item $u \in C^0([0, T];  \ell^1_{x_0,\GM^s}(V))$ for all $x_0\in V$;

\smallskip

\item  $u^m \in L^1( (0, T) ; C(V) ) $; 

\smallskip

\item $u$ satisfies the identity
\begin{equation}\label{def_eq}
\int_{0}^{T}  \partial_t \psi(t) \,  (- \Delta)^{-s} u(t,x)  \, {\rm d}t - \int_{0}^{T}   u^m(t,x)  \, \psi(t) \, {\rm d}t = 0
\end{equation}
for every test function $\psi \in C^1_c(0,T)$ and every $x \in V$;

\smallskip

\item $u(0,\cdot)=u_0$ in $V$.

\end{itemize}

\end{defi}
\begin{rem}\label{Rem_Delta_s_cont}
Equation \eqref{def_eq} makes sense because the function $t\mapsto  (- \Delta)^{-s} u(t,x) $ is integrable on $[0,T]$ since $u \in C^0([0, T];  \ell^1_{x,\GM^s}(V))$ for all $x \in V$. Indeed, for every $t,\tau \in[0,T]$ and every $x\in V$
\begin{equation*}
| (- \Delta)^{-s} u(t,x)- (- \Delta)^{-s} u(\tau,x)| \le \sum_{y\in V} |u(t,y)-  u(\tau,y)| \GM^s(x,y)\deg(y)\leq \|u(t)-u(\tau)\|_{ \ell^1_{x,\GM^s}(V)}\,.
\end{equation*} 
\end{rem}

\begin{rem} Definition \ref{defi_WDS} can be equivalently stated by replacing \eqref{def_eq} with the following equation
  \begin{equation}\label{def_eq2}
\int_{0}^{T} \sum_{x \in V} \partial_t \psi(t,x) \,  (- \Delta)^{-s} u(t,x)  \, \deg(x){\rm d}t - \int_{0}^{T} \sum_{x \in V}  u^m(t,x)  \, \psi(t,x) \, \deg(x){\rm d}t = 0,
\end{equation}
for every test function $\psi \in C^1_c((0,T);C_c(V))$. 

%
%

As for the equivalence of \eqref{def_eq} and \eqref{def_eq2}, observe that taking $ \psi(t,x)=\psi(t) \frac{\delta_{x_0}(x)}{\deg(x_0)}$ in \eqref{def_eq2} immediately yields \eqref{def_eq}. Conversely, under the assumption of Definition \ref{defi_WDS}, testing \eqref{def_eq} with $\psi \in C^1_c((0,T);C_c(V))$ and summing with respect to $x\in V$ one gets \eqref{def_eq2}.

Finally, for future purposes we point that a nonnegative WDS $u$ satisfies the following identity as well: 
\begin{equation}\label{def_eq3}
\int_{0}^{T}\partial_t \psi(t) \sum_{x \in V}  \,  (- \Delta)^{-s} \varphi(x) u(t,x)\, \deg(x)  \, {\rm d}t - \int_{0}^{T} \psi(t)  \sum_{x \in V}  u^m(t,x)  \, \varphi(x)\,  \deg(x)\, {\rm d}t = 0
\end{equation}
for every test function $\psi \in C^1_c(0,T)$ and $\varphi \in  C_c(V)$. 
  To prove \eqref{def_eq3}  we notice that, by exploiting the fact that $ \GM^s(x,y)= \GM^s(y,x)$ for all $x,y \in V$ and since  {$\varphi \in C_c(V)$ and $u \in \ell^1_{x_0,\GM^s}(V)$ for every $x_0 \in V$, by Fubini's theorem} 
  \begin{align*}
   &\sum_{x \in V}  (- \Delta)^{-s} u(t,x)  \varphi(x) \deg(x) =  \sum_{x \in V}      \sum_{y \in V} u (t,y) \, \GM^s(x,y)   \varphi(x) \deg(y) \deg(x)\\
   &= \sum_{y \in V}    u (t,y)  \sum_{x \in V}  \, \GM^s(y,x)   \varphi(x) \deg(x)\deg(y)=\sum_{y \in V}  (- \Delta)^{-s} \varphi(y) u(t,y) \deg(y) \notag
   \end{align*}
Hence, if we take $ \psi(t,x)=\psi(t) \varphi(x)$ in \eqref{def_eq2}, by using the above identity, we get \eqref{def_eq3}.
\end{rem}

 \subsection{{Mild solutions with initial data in $\ell^1$ are WDS}}
 
  { In this section we recall the notion of mild solutions to Problem~\eqref{NFDE} and their main properties. We then show that mild solutions with initial data in $\ell^1$ belong to the class of weak dual solutions. Existence, uniqueness, and comparison principles for mild solutions follow from the classical theory developed by Bénilan, Crandall, Pazy, and Pierre; see \cite{BCr,BCPbook,CP}.}

 \begin{defi}\label{defmild}
Let $ u_0 \in \ell^1(V)$ with $u_0\geq 0$. We say that a nonnegative measurable function $u$ is a mild solution to Problem~\eqref{NFDE} if $ u \in C^0([0,T];\ell^1(V)) $ and, for every $T>0$ and $ n \in  {\mathbb{N}_+} $ there exists a piecewise-constant function $u_n$ in $\ell^1(V)$ defined by 
$$
u_n(0)=u_0, \qquad	u_n(t) = u_k \quad \text{if }  t_{k-1} < t \le t_{k}   {\hbox{ and } 1 \le k \le n,}
	$$
	where $t_k = \frac{k}{n} \, T$, for any integer $ 0 \leq k \leq n$, $u_k $ is the solution to the following fractional elliptic equation for $ 0 \leq k \leq n-1$
	\begin{equation}\label{eq-approximate}
	h \, (-\Delta)^s \!\left(u_{k+1}\right)^m    +  u_{k+1} = u_{k}  \qquad \mbox{in} \ \hn \, ,
	\end{equation}
	$h =   \frac{T}{n}$, 
	and 
	$$
	\lim_{n\rightarrow \infty}\sup_{t\in [0,T]}\left\| u_n(t) -u(t)\right\|_{\ell^1(V)}  =0.
	$$
\end{defi}

\begin{prop}\label{BCPP-Thm}
	Let $u_0,v_0\in \ell^1(V)  $, with $ u_0,v_0 \ge 0 $. Then there exist unique nonnegative mild solutions $u,v\in C^0([0,+\infty); \ell^1(V))$ to Problem \eqref{NFDE} corresponding to the initial data $u_0,v_0$, respectively, such that
		\begin{equation}\label{comparison.L1}
	\sum_{x\in\hn} \left( u(t,x)-v(t,x)\right)_+ \dg(x) \le \sum_{x\in\hn} \left( u_0(x)-v_0(x)\right)_+ \dg(x) \qquad \mbox{for all } t\ge 0\,.
	\end{equation}
	In particular,
	\begin{equation*}
	\left\| u(t)-v(t) \right\|_{\ell^1(V)}\le \left\| u_0-v_0 \right\|_{\ell^1(V)} \qquad \forall t\ge 0 \, .
	\end{equation*}
	Moreover, the nonnegative mild solutions enjoy the time monotonicity property:
	\begin{equation}\label{mon-est.OLD}
	\mbox{the map} \quad t \mapsto t^{\frac{1}{m-1}} u(t,x) \quad \mbox{is nondecreasing for all $x \in V $.}
	\end{equation}
\end{prop}

\begin{proof} 

As already remarked, the statement follows from the classical theory based on nonlinear semigroup. Indeed, we are dealing with an equation of the form $\partial_t u+ A \varphi(u)=0$, where the operator $A: dom(A) \subset \ell^1(\hn)\to \ell^1(\hn)$ is a densely defined linear $m$-accretive operator of sub-Markovian type.

 In our case, $\varphi(r)=r^m$ ($m>1$) is the standard example satisfying the assumptions in \cite{CP} while $A=(-\Delta)^s$ and satisfies assumptions (A1)--(A2) in \cite{CP} as well.
 

Then, by Proposition~1 and Proposition~2 in \cite{CP}, the nonlinear operator $A_\varphi$ associated with $A\varphi$ is $m$-accretive in $\ell^1(V)$. As a consequence, from the Crandall--Liggett theorem (see e.g, \cite[Theorem 10.16]{V}), $A_\varphi$ generates a nonlinear contraction semigroup $(S(t))_{t\ge 0}$ on $\ell^1(V)$, and for every nonnegative $u_0 \in \ell^1(V)$ there exists a unique mild solution
\[
u(t) = S(t)u_0 \in C([0,+\infty);\ell^1(V)).
\]

By Proposition~1(iv) in \cite{CP}, the resolvent $(I+\lambda A_\varphi)^{-1}$ is order-preserving and nonexpansive in $\ell^1(V)$. Passing to the semigroup, this implies that $S(t)$ is order-preserving and $T$-contractive. Hence, if $u$ and $v$ are the mild solutions corresponding to nonnegative initial data $u_0,v_0 \in \ell^1(V)$, we have \eqref{comparison.L1}. Finally, the time monotonicity property \eqref{mon-est.OLD} follows from Theorem~2 in \cite{BCr}. \end{proof}

  {The next lemma shows that, under suitable assumptions, $\left( -\Delta \right)^{-s}$ acts as a left inverse of $\left( -\Delta \right)^{s}$, a property that will be crucial in the proof of Proposition \ref{thm-weak-mild}.}

 \begin{lem}\label{lemma-inv}
\ Let  $1< p  < 2$ and let $n>2$ be such that \eqref{Nash} holds.   {Furthermore, assume that \(0 < s \le 1\) and \(s <  {n}\left( \frac{1}{p} - \frac{1}{2}\right)\).}   Then, the operator $ (-\Delta)^{-s} $ is continuous from $ \ell^p(V) $ to $\ell^{p'}(V)$, where $p'$ is the conjugate exponent of $p$. Moreover, if    {$u \in \ell^p(\hn)$}, the following left-inverse formula holds: 
\begin{equation}\label{left-infv}
\left( -\Delta \right)^{-s}\left[ \left( -\Delta \right)^{s} u \right] = u \qquad  \text{in $V$}.
\end{equation}

\end{lem}
{ \begin{proof}
 We start by noticing that the estimate \eqref{stimaheat} can be used to derive strong $(p,p')$ bounds
    for the heat semigroup. Interpolating between \eqref{stimaheat} and the
    trivial $(2,2)$-boundedness, we obtain 
    \[
        \|e^{t\Delta}\|_{\ell^p(V)\to \ell^{p'}(V)}
        \le
        C\,\min\{1,t^{- {n\left( \frac{1}{p} - \frac{1}{2}\right)}}\}     \qquad \text{for every} \   1\le p\le 2.
    \]
    Consequently, for every $s\in(0,1]$ the operator
    $(-\Delta)^{-s}:\ell^p(V)\to \ell^{p'}(V)$ is well defined  and bounded. Indeed,
    \begin{align*}
        \|(-\Delta)^{-s}u\|_{\ell^{p'}(V)}
        &\le
        \int_0^1 \|e^{t\Delta}u\|_{\ell^{p'}(V)}\,\frac{dt}{t^{1-s}}
        +
        \int_1^{\infty} \|e^{t\Delta}u\|_{\ell^{p'}(V)}\,\frac{dt}{t^{1-s}} \\
        &\le
        C\int_0^1 \,\|u\|_{\ell^p(V)}\,\frac{dt}{t^{1-s}}
        +
        C\int_1^{\infty}
        \|u\|_{\ell^p(V)}\,t^{- {n\left( \frac{1}{p} - \frac{1}{2}\right)}}
        \frac{dt}{t^{1-s}} \le C_{p,n,s}\|u\|_{\ell^p(V)},
    \end{align*}
   where we have exploited the fact that the integral is finite since $p<2$ and
    $s< n\left(\frac{1}{p}-\frac{1}{2}\right) $.
    
    As for the proof of \eqref{left-infv}, let $(T_t^{ {(s)}})_{t>0}=(e^{-t(-\Delta)^s})_{t>0}$ denote the subordinated semigroup generated by $ (-\Delta)^{s}$ which is submarkovian and hence contractive on $\ell^p(V)$
for every $1\le p\le\infty$. In order to prove \eqref{left-infv} we first show that for every $1< p<\infty$ and $s \in (0,1]$ 
\begin{equation}\label{T_tzero}
\|T_t^{ {(s)}} u\|_{\ell^p(V)} \xrightarrow[t\to\infty]{} 0
\qquad \text{for all } u\in \ell^p(V).
\end{equation}
We begin showing \eqref{T_tzero} for $s=1$  {and in this case we simply write $T_t$ in place of $T_t^{(1)}$}. Assume that $u \in C_c(V)$. Observe that 
\begin{align*}
   |T_tu(x)| \le \sum_{y \in \mathrm{supp} \, u} h_t(x,y)|u(y)| \mathrm{deg}(y) \le \|u\|_\infty\sum_{y \in \mathrm{supp}\, u} h_t(x,y) \mathrm{deg}(y),
\end{align*} and the right hand side tends to zero as $t \to \infty$ by \eqref{ultracontr1} because the sum has finite summands. Set $Mu(\cdot)=\sup_{t>0} |T_t u(\cdot)|$, and recall that $M$ is bounded on $\ell^p(V)$ (see \cite[Chapter 3]{Stein}). Observe that
\begin{align*}
    \sum_{x \in V} |T_tu(x)|^p \mathrm{deg}(x) \le  \sum_{x \in V} |Mu(x)|^p \mathrm{deg}(x) \le C \sum_{x \in V} |u(x)|^p \, \mathrm{deg}(x)<\infty.
\end{align*}
Then, by dominated convergence, \begin{align*}
\lim_{t \to \infty}   \sum_{x \in V} |T_tu(x)|^p \mathrm{deg}(x) =   \sum_{x \in V} \lim_{t \to \infty} |T_tu(x)|^p \mathrm{deg}(x) =0,
\end{align*}
as desired.
For a general $u \in \ell^p(V)$, observe that for every $\varepsilon>0$ there is a $u_{{n}} \in \ell^p(V)\cap C_c(V)$ such that $\|u_n-u\|_{\ell^p(V)}<\varepsilon$, so that
\begin{align*}
    \| T_t u\|_{\ell^p(V)} \le \|T_t u_n\|_{\ell^p(V)}+\|T_t (u-u_n)\|_{\ell^p(V)} \le \|T_tu_n\|_{\ell^p(V)}+\varepsilon,
\end{align*} and the right hand side tends to $\varepsilon$ as $t\to \infty$, proving \eqref{T_tzero} for $s=1$.\\
 We now prove  \eqref{T_tzero}  for $s\in(0,1)$. Let $K_{t,s}$ denote the kernel of $T_t^{ {(s)}}$. It can be written  by subdordination (see \cite[Chapter IX]{Yosida80}) as
\begin{align*}
    K_{t,s}(x,y)=\int_0^{\infty} h_\tau(x,y) \eta_{t,s}(\tau) \, \mathrm{d}\tau,
\end{align*} where $\eta_{t,s}$ satisfies
\begin{align*}
    \eta_{t,s}(\tau) \ge0, \qquad \int_0^\infty \eta_{t,s}(\tau)\, \mathrm{d}\tau=1 \quad \forall t>0, \qquad 
    \int_0^\infty e^{-\lambda \tau}\eta_{t,s}(\tau) \, \mathrm{d\tau}=e^{-t\lambda ^s} \quad   \forall \lambda, t>0.
\end{align*} We claim that 
\begin{align*}
    \lim_{t \to \infty} K_{t,s}(x,y)=0 \qquad\forall x,y \in V.
\end{align*} This would allow us the repeat the proof used for $s=1$. \\  Fix $R>0$ big enough. Then,
\begin{align*}
    K_{t,s}(x,y)&=\int_0^R h_\tau(x,y)\eta_{t,s}(\tau) \, \mathrm{d\tau}+\int_R^\infty  h_\tau(x,y) \eta_{t,s}(\tau) \, \mathrm{d\tau} \\
    &\le    {\max_{\tau \in [0, R]}h_\tau(x,y)} \int_0^R \eta_{t,s}(\tau) \, \mathrm{d\tau}+\sup_{\tau \ge R} h_\tau(x,y) \int_R^\infty \eta_{t,s}(\tau) \, \mathrm{d\tau}.
\end{align*} 
Moreover,
\begin{align*}
    \int_0^R \eta_{t,s}(\tau) \, \mathrm{d\tau} \le e^{\lambda R} \int_0^R e^{-\lambda \tau} \eta_{t,s}(\tau) \, \mathrm{d\tau} \le e^{\lambda R}  e^{-t\lambda^s},
\end{align*}so that $$\lim_{t \to \infty}\int_0^R \eta_{t,s}(\tau) \, \mathrm{d\tau}=0.$$
 On the other hand, by \((2.5)\), for every $\varepsilon>0$ there exists $R>0$ such that  $|h_\tau(x,y)|<\varepsilon$ for every $\tau>R$.
Furthermore,
\[
\int_R^\infty \eta_{t,s}(\tau)\,\mathrm{d}\tau
\le
\int_0^\infty \eta_{t,s}(\tau)\,\mathrm{d}\tau
=1.
\]
Hence,
\[
\limsup_{t\to\infty} K_{t,s}(x,y)\le \varepsilon.
\]
Since \(\varepsilon>0\) is arbitrary, this proves the claim for every $s\in (0,1]$.
 
  { Coming back to the proof of \eqref{left-infv}, the limit proved in \eqref{T_tzero} allows to apply  \cite[Proposition 6.2.12]{J-3} and conclude that
  $$\lim_{h \to 0^+}  \int_0^{+\infty} T_t^{ {(s)}} \left[  \frac{u - T_h^{ {(s)}} u}{h}\right]\,  \mathrm{d}t=u\,.$$
 The above limit, suitable combined with the following known identities:
$$
\lim_{h \to 0^+} \frac{u - T_h^{ {(s)}} u}{h} = \left( -\Delta \right)^s u \quad \text{in } \ell^p(\hn) \qquad \text{and} \qquad
(-\Delta)^{-s} = \int_0^{+\infty} T_t^{ {(s)}} \, \mathrm{d}t
$$
 {and the continuity of $(-\Delta)^{-s}$} yield the  {desired conclusion}.}
   \end{proof}

  {The next result provides the link between the classical semigroup approach and the weak dual formulation by showing that mild solutions with $\ell^1(\hn)$ initial data are WDS.}

\begin{prop}\label{thm-weak-mild}
Let $u$ be the mild solution to \eqref{NFDE} corresponding to any nonnegative initial datum $u_0 \in \ell^1(\hn)$. Then, $u$ is a WDS in the sense of Definition~\ref{defi_WDS}.
\end{prop}
\begin{proof}
  {As already recalled, mild solutions are obtained as limits of an implicit time discretization scheme (see Definition~\ref{defmild}). 
Fix \(T>0\) and \(n\in\mathbb{N}\), and let \(u_n\) be the corresponding piecewise-constant interpolation defined in Definition~\ref{defmild}. 
Then the mild solution \(u \in C^0([0,T];\ell^1(V))\) is obtained as the uniform limit of \(u_n\) in \(\ell^1(V)\) (see \cite[Theorem 10.16]{V}).
By construction, the sequence \(\{u_k\}_k\) satisfies \eqref{eq-approximate}, and, by induction, \(u_k \in \ell^1(V)\) together with the nonexpansivity estimate
\[
\| u_{k+1} \|_{\ell^1(V)} \le \| u_k \|_{\ell^1(V)} \le \cdots \le \| u_0 \|_{\ell^1(V)}.
\]
Consequently,
\[
\| u_n(t) \|_{\ell^1(V)} \le \| u_0 \|_{\ell^1(V)} \qquad \forall t \in [0,T].
\]}

Let $\psi \in C^1_c( (0, T))$, we set $ \psi_k = \psi(t_k) $ and we fix $x\in V$. 
Observe that since $u_{k+1} \in \ell^1(V)$, we have that $u_{k+1}^m \in \ell^1(V) \subset \ell^p(V)$ for every $p \in [1,\infty]$. Moreover, by Lemma \ref{lemma-inv}, we obtain
\[
(-\Delta)^{-s} (-\Delta)^s u_{k+1}^m = u_{k+1}^m \qquad \text{in }V.
\]
In view of the above, by applying $(-\Delta)^{-s}$ to both members of \eqref{eq-approximate}, multiplying by $\psi_k$ and summing all terms over $ k=0,\ldots,n-1 $, we get  

$$	\sum_{k =0}^{n-1}  \left[ (-\Delta)^{-s} u_{k+1}(x) - (- \Delta)^{-s} u_k(x) \right] \psi_{k} \, 
	= - h \, \sum_{k=0}^{n-1}  \left(u_{k+1}(x)\right)^m \psi_{k}  \, ,
$$
	which can be rewritten as
$$
 [\psi_{n-1} \, (-\Delta)^{-s} u_n(x) \, - \, \psi_0 \, (-\Delta)^{-s} u_0(x)] -  I_1 = {-}I_2 
$$
where 
$$I_1=\sum_{k =1}^{n-1}  \left(\psi_k - \psi_{k-1}\right) (-\Delta)^{-s} u_k(x) \quad \text{and}  \quad I_2=h \sum_{k =0}^{n-1}  (u_{k+1}(x))^m \, \psi_k.$$
	As concerns $ I_1 $, we have 
	\begin{align*}
	I_1 & =  \sum_{k =1}^{n-1}  h \,  \frac{\psi_k - \psi_{k-1}}{h} \, (-\Delta)^{-s} u_k (x)\,  \label{e012}  \\
	& = \sum_{k =1}^{n-1}  h \,  \partial_t \psi ({t}_k) \, (-\Delta)^{-s} \!\left( u_k - \tilde{u}_k\right)  (x)
	+  \sum_{k =1}^{n-1}  h \,  \partial_t \psi ({t}_k) \, (-\Delta)^{-s}  \tilde{u}_k(x) \, + \rho_n (x) \, ,\nonumber
	\end{align*}
	where we set  $  \tilde{u}_k = u(t_k ,\cdot)$, and
	$$
	\rho_n(x) =  \sum_{k =1}^{n-1}  h \,  \left[ \frac{\psi_k - \psi_{k-1}}{h} - \partial_t \psi ({t}_k) \right] (-\Delta)^{-s} u_k(x) \, \, .
	$$
	Since $ \psi $ belongs to $ C^1_c( (0, T))$, $ \{ u_k \} $ is uniformly bounded in $ \ell^1(V)$  {and by the $(1,\infty)$-boundedness of $(-\Delta)^s$}, we have that
	$$
	|(-\Delta)^{-s} u_k(x)|\leq \|(-\Delta)^{-s} u_k\|_{{\ell^\infty(V)}}\le  {C} \|u_k\|_{ {\ell^1(V)}}\leq  {C}\|u_0\|_{ {\ell^1(V)}},
	$$
	hence it is readily seen that $ \rho_n(x) \to 0 $ as $ n \to \infty $. Let us then focus on the other summands of $I_1$. As for the first one, 
	 {using again the $(1,\infty)$-boundedness} of $(-\Delta)^{-s}$ we have 
	\begin{align*}
 &\left| \sum_{k =1}^{n-1}  h \,  \partial_t \psi ({t}_k,\cdot) \, (-\Delta)^{-s} \!\left( u_k - \tilde{u}_k\right) (x)  \right|
\leq \|\partial_t \psi \|_{C[0,T]}\,  \sum_{k =1}^{n-1}  h  \|  (-\Delta)^{-s} \left(  u_k - \tilde{u}_k \right)\|_{\ell^{{\infty}}(V)}\\
	& \leq  \, C  T \sup_{t \in [0,T]} \left\| u_n(t) - u(t) \right\|_{\ell^{{1}}(V)}  
	   \underset{n \to \infty}{\longrightarrow} 0 \, .
	\end{align*}
	On the other hand, the second summand in $I_1$ is a Riemann sum so that, using the continuity of $ \partial_t \psi $ and of the function $u$ with respect to $t$ we get
	$$
	\sum_{k =1}^{n-1}  h \, \partial_t \psi ({t}_k) \, (-\Delta)^{-s}  \tilde{u}_k (x)
	\underset{n \to \infty}{\longrightarrow}  \int_0^T  \partial_t \psi (t)\, (-\Delta)^{-s} u(t,x) \,  \, {\rm d}t \, .
	$$
     Hence, by combining the above results, we deduce that
	\begin{equation}\label{Riemann-1}
	I_1 \underset{n \to \infty}{\longrightarrow}  \int_0^T  \partial_t \psi(t) \, (-\Delta)^{-s} u (t,x) \, {\rm d}t \, .
	\end{equation}
	As far as $I_2$ is concerned, we have that 
	\begin{equation*}
	I_2 =  h\sum_{k =0}^{n-1}[u_{k+1}^m(x)  - \tilde u_k^m(x)] \, \psi_k +h\sum_{k =0}^{n-1}\tilde u_k^m(x)   \, \psi_k\, .
	\end{equation*}
	 For the first summand, exploiting the fact that also $ \{ u^m_n \} $ converges to $ u^m $ uniformly in $  \ell^1(V)$ and $ u \in C^0([0,T];\ell^1(V)) $, we have that
	\begin{align*}
 &\left| h\sum_{k =0}^{n-1}[u_{k+1}^m(x)  - \tilde u_k^m(x)] \, \psi_k \right| \leq \frac{h}{d_0} \| \psi \|_{C[0,T]} \sum_{k =0}^{n-1}\|u_{k+1}^m  - \tilde u_k^m\|_{\ell^1(V)}\\
 & \leq  \frac{h}{d_0} \| \psi \|_{C[0,T]} \left(\sum_{k =0}^{n-1} [\|u_{n}^m(t_{k+1}, \cdot)  -  u^m(t_{k+1}, \cdot)\|_{\ell^1(V)}+\| u^m(t_{k+1}, \cdot)  - u^m(t_{k}, \cdot)\|_{\ell^1(V)} ]\right)\\
 & \leq   \frac{T}{d_0} \| \psi \|_{C[0,T]}\left( \sup_{t \in [0,T]} \left\| u_n^m(t) - u^m(t) \right\|_{\ell^1(V)}+  \sup_{k=0,...,n-1} \left\| u^m(t_{k+1}, \cdot)  - u^m(t_{k}, \cdot) \right\|_{\ell^1(V)} \right) \underset{n \to \infty}{\longrightarrow} 0 \, .
 \end{align*}
	As for the second summand in $I_2$, it is a Riemann sum so that, using the continuity of $ \psi $ and of the function $u$ with respect to $t$ we get
	$$
	h\sum_{k =0}^{n-1}\tilde u_k^m(x)   \, \psi_k
	\underset{n \to \infty}{\longrightarrow}  \int_0^T  \psi (t)\,  u^m(t,x) \,  \, {\rm d}t \, .
	$$
	Therefore,
	\begin{equation}\label{Riemann-2}
	I_2 \underset{n \to \infty}{\longrightarrow}  \int_0^T u^m \, \psi \, {\rm d}t \, .
	\end{equation}
	Since $\psi$ is compactly supported in time, $ \psi_0 =0  $ and $ \psi_{n-1}=0 $ for sufficiently large $n$, whence
	\begin{equation}\label{Riemann-3}
 \, [\psi_{n-1} \, (-\Delta)^{-s} u_n  -  \psi_0 \, (-\Delta)^{-s} u_0] \,   \underset{n \to \infty}{\longrightarrow}   0 \, .
	\end{equation}
	As a consequence of \eqref{Riemann-1}, \eqref{Riemann-2} and \eqref{Riemann-3}, we finally obtain
	$$
	 \int_0^T  \partial_t \psi(t) \, (-\Delta)^{-s} u(t,x)  \, {\rm d}t
	- \int_0^T   u^m(t,x) \, \psi(t) \, {\rm d}t = 0\, .
	$$
Given the arbitrariness of $ T $ and $\psi$, the above identity shows that $u$ is indeed a WDS, recalling moreover that since $\ell^1(\hn)$ is included (with continuity) in  $ \ell^1_{x_0,\GM^s}(\hn) $ for every $x_0 \in  V$,   $u$ also belongs to $ C^0([0,+\infty);\ell^1_{x_0,\GM^s}(\hn)) $.
\end{proof}

\subsection{  {Weighted $\ell^1$ estimates for nonnegative WDS}} 
In this subsection we establish weighted $\ell^1$ a priori estimates for nonnegative weak dual solutions. These estimates will play a crucial role in the existence theory developed in the next subsection.

\begin{prop}\label{crucial}
Let $u$ be a WDS to \eqref{NFDE} corresponding to a nonnegative initial datum    $u_0 \in \ell^1_{\GM^s}(V)$. Then, we have
\begin{equation}\label{estimate-1}
 \sum_{x \in V} u(t, x)  \, \GM^s (x, x_0) \deg(x)\, \leq \sum_{x \in V} u_0(x) \,  \GM^s(x, x_0) \deg(x)\,  \qquad \text{for all $t \ge 0 $ and all $ x_0 \in V $} \, ,
\end{equation}
namely $$\left\| u(t)  \right\|_{\ell^1_{x_0,\GM^s}(V)} \leq  \left\| u_0  \right\|_{\ell^1_{x_0,\GM^s}(V)}  \qquad \text{for all $t \ge 0$ and all $ x_0  \in V $} \, ,$$
and
\begin{align}\label{main-estimate-1}
 \left( \frac{t_0}{t_1} \right)^{\frac{m}{m-1}} (t_1 - t_0)\, u^m (t_0, x_0)
 &\leq
  \sum_{x \in V} \left[ u(t_0, x)-u(t_1, x) \right] \GM^s(x, x_0)\deg(x) \, \leq (m-1) \, \frac{t^{\frac{m}{m-1}}}{t_0^{\frac{1}{m-1}}} \, u^m(t, x_0)
\end{align}
for all  $0<t_0\leq t_1\leq t$ and all $ x_0  \in V $. 
\end{prop}

\begin{proof}

The idea of the proof is the same of \cite[Proposition~4.2]{BV2} and \cite[Proposition 3.3]{BBGG} in the continuous setting, where bounded Euclidean sets and the hyperbolic space are treated, respectively. 
 {Fix $t>0$ and let $0<t_0\leq t_1\leq t$.} 
  {Choosing in \eqref{def_eq3}  as test functions a sequence \( \{\psi_n\}_n \subset C_c^1((0,+\infty)) \) 
such that \( 0 \le \psi_n \le 1 \), 
\( \mathrm{supp}\,\psi_n \subset (t_0 - \tfrac{1}{n},\, t_1 + \tfrac{1}{n}) \subset (0,t) \) for $n$ large enough, 
\( \psi_n \to \chi_{[t_0,t_1]} \) a.e. in \( (0,+\infty) \), and 
\( \partial_t \psi_n \to \delta_{t_0} - \delta_{t_1} \) in the sense of distributions, 
as \( n \to \infty \), the following identity holds:}
\begin{align}\label{limit-1_new}
&\sum_{x \in V} u(t_0, x) \, (- \Delta)^{-s} \varphi(x)\deg(x) \,
- \sum_{x \in V} u(t_1, x) \, (- \Delta)^{-s} \varphi(x)\deg(x)\\
& =  \int_{t_0}^{t_1} \sum_{x \in V}  u^{m}(\tau, x) \, \varphi(x)\deg(x) \, \, {\rm d}\tau \ge 0 \notag
\end{align}
for every nonnegative $\varphi \in C_c(V)$. Finally, by taking $\varphi(x)= \frac{\delta_{x_0}(x)}{\deg(x_0)}$, $t_1=t$ and  by letting $t_0\to 0$, we get \eqref{estimate-1}.
 \\
The proof of \eqref{main-estimate-1} follows by \eqref{limit-1_new} plus the monotonicity property \eqref{mon-est.OLD}. More precisely, we consider the sequence $\psi_{n} \in C_c^{\infty}(0, + \infty)$ defined above. Then, for all nonnegative $\varphi \in C_c(V)$, by using that that the map  $t \mapsto t^{\frac{m}{m-1}} u^m$ is nondecreasing, we get
\begin{align*}
& \int_0^{\infty} \psi_{n}(\tau) \sum_{x \in V}  \, u^m(\tau, x) \, \varphi(x) \, \deg(x) \, {\rm d}\tau\\
&\leq \int_{t_0 -\frac{1}{n}}^{t_1 + \frac{1}{n}}  \, \left( \frac{t}{\tau} \right)^{\frac{m}{m-1}}
\sum_{x \in V}  u^m (t, x) \, \varphi(x) \, \deg(x) \, {\rm d}\tau \\
& \leq  \,\left( \int_{t_0 -\frac{1}{n}}^{t_1 + \frac{1}{n}} \,
\left( \frac{t}{\tau} \right)^{\frac{m}{m-1}} \, {\rm d}\tau \right) \,\left( \sum_{x \in V}  u^m(t, x) \, \varphi(x) \,\deg(x) \right) \\
& \leq (m-1) \frac{1}{(t_0 - \frac{1}{n})^{\frac{1}{m-1}}} \, t^{\frac{m}{m-1}} \,
\sum_{x \in V}  \, u^m(t, x) \, \varphi(x) \, \deg(x).
\end{align*}
Now, letting $n \rightarrow \infty$, we get
$$
\int_{t_0}^{t_1} \sum_{x \in V} \, u^m(\tau, x) \, \varphi(x) \, \deg(x) \, {\rm d}\tau
 \leq  \frac{(m-1)}{t_0^{\frac{1}{m-1}}} \, t^{\frac{m}{m-1}} \, \sum_{x \in V} u^m(t, x) \, \varphi(x) \, \deg(x),
$$
for all $t_0 \leq t_1 \leq t.$ Finally, combining the above inequality with \eqref{limit-1_new} and testing with $\varphi(x)= \frac{\delta_{x_0}(x)}{\deg(x_0)}$ we obtain the upper bound estimate of  \eqref{main-estimate-1}. The lower bound follows similarly.\end{proof}

Minor changes in the proof of Proposition \ref{crucial} also give the following.

\begin{prop}\label{monotonePhi}
	If $ u,v $ are two \emph{ordered} WDS to Problem \eqref{NFDE} corresponding to nonnegative initial data   {$u_0,v_0 \in \ell^1_{\GM^s}$}, respectively, it holds
\begin{align}\label{newmonotonicity2-bis}
\left\| u(t) - v(t) \right\|_{\ell^1_{x_0,\GM^s}(V)} \leq  \left\| u_0 - v_0 \right\|_{\ell^1_{x_0,\GM^s}(V)}  \qquad \text{for all $t \ge 0$ and all $ x_0  \in \hn $} \, .
\end{align}

\end{prop}

\begin{proof}
Assume $u\leq v$, by writing \eqref{limit-1_new} for $u$ and $v$ and taking the difference we immediately get
\begin{align*}
&\sum_{x \in V} (u(t_0, x)-v(t_0, x)) \, (- \Delta)^{-s} \varphi(x)\deg(x) \,
- \sum_{x \in V} (u(t_1, x)-v(t_1, x)) \, (- \Delta)^{-s} \varphi(x)\deg(x)\\
& =  \int_{t_0}^{t_1} \sum_{x \in V}  (u^{m}(t, x)-v^{m}(t, x)) \, \varphi(x)\deg(x) \, \, {\rm d}t \ge 0, \notag
\end{align*}
for every nonnegative $\varphi \in C_c(V)$, and all $ 0 <t_0<t_1 $.  Therefore, by taking $\varphi(x)= \frac{\delta_{x_0}(x)}{\deg(x_0)}$, $t_1=t$ and  by letting $t_0\to 0$ we get \eqref{newmonotonicity2-bis}.

\end{proof}

\subsection{Existence of WDS in $\ell^1_{\GM^s}$}\label{subsectionexistence}
We now construct WDS corresponding to nonnegative initial data in $\ell^1_{\GM^s}(V)$ as monotone limits of solutions associated with $\ell^1(V)$ data.

\begin{thm}\label{existenceWDS}
	Let $u_0 $ be any nonnegative initial datum such that $ u_0\in  \ell^1_{\GM^s}(V)$. Then there exists a WDS to Problem~\eqref{NFDE}, in the sense of Definition \ref{defi_WDS}.
\end{thm}

\begin{proof}

Let $u_0 \in \ell^1_{\GM^s}(V)$ be a nonnegative datum. Following a strategy commonly used in the literature about WDS in the continuous setting (see e.g., \cite[Theorem 2.4]{BBGM}), we construct a WDS corresponding to $u_0$ by (monotone) approximation in terms of the mild solutions corresponding to $ \ell^1(V)$ initial data, which by Proposition \ref{thm-weak-mild} are also WDS. More precisely, to $u_0$ we associate the sequence 
\begin{equation*}\label{monotone-approx-1}
u_{0,n} = \chi_{K_n} u_0  \qquad \text{for } n \in \mathbb{N} \,,
\end{equation*}
where $\{K_n\}_n$ is an increasing sequence of compact sets such that $\bigcup_n K_n =V$.

 Since each $ u_{0,n} $ is compactly supported, $u_{0,n}\in \ell^1(\hn) $ and, by construction, $u_{0,n} \leq u_{0,n+1} \leq u_{0} $. If we let $ u_n = u_n(t,x) $ denote the (mild and) WDS to \eqref{NFDE} with initial datum $ u_{0,n} $, from \eqref{comparison.L1} we deduce that $u_n(t,x)\le u_{n+1}(t,x)$ for every $ t>0 $ and $ x \in \hn $. Therefore, the sequence of solutions $ \{ u_n \} $ is monotone increasing and the following pointwise limit exists
\begin{equation*}
u(t,x)=\lim_{n\to\infty}u_{n}(t,x)\,,
\end{equation*}
for every $ t>0 $ and $ x \in \hn $. We prove that $u$ is a WDS with corresponding initial datum $u_0 $.

We start by showing that, $ u_{n}(t,\cdot) \to u(t,\cdot) $ in $ C^0([0,T]; \ell^1_{x_0,\GM^s}(V)) $ as $ n \to \infty $,  for all $ x_0 \in V $. To this aim we first notice that Proposition \ref{crucial} yields
$$
 \sum_{x \in V} u_n(t, x)  \, \GM^s (x, x_0) \deg(x)\, \leq \sum_{x \in V} u_{0,n}(x) \,  \GM^s(x, x_0) \deg(x)\,  \qquad \text{for all $t \ge 0 $ and all $ x_0 \in V $} \, .
$$
By the monotone convergence theorem, $ u_{0,n} \to u_0 $ in $ \ell^1_{x_0,\GM^s}(V) $ as $ n \to \infty $, we obtain
\begin{align*}
 \sum_{x \in V} u(t, x)  \, \GM^s (x, x_0) \deg(x)\, \leq \sum_{x \in V} u_0(x) \,  \GM^s(x, x_0) \deg(x)\,  \qquad \text{for all $t \ge 0 $ and all $ x_0 \in V $} 
\end{align*}
which in particular yields
\begin{align}\label{stability_G2}
 \left\| u(t)\right\|_{\ell^1_{x_0,\GM^s}(V) }\leq   \left\| u_0\right\|_{\ell^1_{x_0,{\GM^s}}(V)}  \qquad \text{for all $t \ge 0 $ and all $ x_0 \in V.$} 
\end{align}
  {Notice that we cannot apply Proposition \ref{crucial} directly to obtain the above estimate, since we do not yet know that $u$ is a WDS.} Hence, $ u(t,\cdot) \in \ell^1_{x_0,\GM^s}(V) $ for all $t>0$ and all $ x_0 \in V$. In particular, by monotone convergence, $ u_{n}(t,\cdot) \to u(t,\cdot) $ in $ \ell^1_{x_0,\GM^s}(V) $ as $ n \to \infty $. 
Finally, by repeating the same arguments outlined in  the proof of \cite[Theorem 2.4]{BBGM} which rely on the exploitation of the stability estimate \eqref{newmonotonicity2-bis} combined with the fact that, for $n$ fixed, the maps $[0,T]\ni t\mapsto \|u_n(t,\cdot)\|_{\ell^1_{x_0,\GM^s}(V)}$ are uniformly continuous, one may show that the sequence $ \{ u_n \} $ is uniformly equicontinuous in $ C^0([0,T]; \ell^1_{x_0,\GM^s}(V)) $, for every fixed $ T>0 $. Then, by Ascoli-Arzelà theorem it follows that $ u_{n}(t,\cdot) \to u(t,\cdot) $ in $ C^0([0,T]; \ell^1_{x_0,\GM^s}(V)) $ as $ n \to \infty $ and $ u \in C^0([0,T]; \ell^1_{x_0,\GM^s}(V))  $ as required in Definition \ref{defi_WDS}.

Now we prove that $  u_n \to u $ also in $L^m((0,T); C(V)) $ as $ n \to \infty $.  From \eqref{stability_G2}, for all $x \in V$ we have that
$$u(t,x)\leq \frac{\left\| u_{0}\right\|_{\ell^1_{\GM^s}(V)}}{d_0 \GM^s(x, o)}\,.$$
In turn, for every $T>0$  we deduce that $ u \in L^m((0,T); C(V))  $.  Again, by monotonicity, this shows that $  u_n \to u $ as $ n \to \infty $ in $L^m((0,T); C(V)) $.

Finally, we show that $u$ satisfies \eqref{def_eq2}.

Since $ u_n$ is a WDS to \eqref{NFDE} for all $n\in \N$ (with initial datum $ u_{0,n} $), it satisfies
  \begin{equation}\label{u_nWDS}
\int_{0}^{T}  \partial_t \psi(t) \,  (- \Delta)^{-s} u_n(t,x)  \, {\rm d}t - \int_{0}^{T}   u_n^m(t,x)  \, \psi(t) \, {\rm d}t = 0
\end{equation}
for every $T>0, x \in V$ and every test function $\psi \in C^1_c(0,T)$.

Thanks to the convergence of $  u_n \to u $ in $L^m((0,T); C(V))$, we get
$$
\int_{0}^{T} u_n^m(t,x)  \, \psi(t) \,{\rm d}t  \underset{n \to \infty}{\longrightarrow} \int_{0}^{T}   u^m(t,x)  \, \psi(t) \, {\rm d}t 
$$
for every $ T>0, x \in V$ and every test function $\psi \in C^1_c(0,T)$.

As for the first term, we already know that $ u \in C^0([0,T]; \ell^1_{x_0,\GM^s}(V)) $; by this, from Remark \ref{Rem_Delta_s_cont} we infer that $(-\Delta)^{-s} u(t,\cdot) \in   C^0([0,T]; C(V))$. Since, by definition of $  (- \Delta)^{-s} $,  also the sequence $ \{ (- \Delta )^{-s} u_n \} $ is monotone increasing, the convergence in $ L^1((0,T); C(V)) $ is assured and we conclude that 
$$
\int_{0}^{T} \partial_t \psi(t) \,  (- \Delta)^{-s} u_n(t,x)  \, {\rm d}t   \underset{n \to \infty}{\longrightarrow} \int_{0}^{T}  \partial_t \psi(t) \,  (- \Delta)^{-s} u(t,x)  \, {\rm d}t  
$$
for every $ T>0, x \in V $ and every test function $\psi \in C^1_c(0,T)$.
The two limits above combined with \eqref{u_nWDS} yield the result.

\end{proof}

\begin{rem}\label{minimal WDS}
By exploiting the stability inequality provided by Proposition~\ref{monotonePhi} and arguing as in the proof of Theorem~4.5 in \cite{BV1}, one can show that the WDS constructed in Theorem~\ref{existenceWDS}, usually referred to in the literature as the \emph{minimal} WDS, is independent of the chosen monotone approximating sequence of initial data. In this sense, such a solution is unique within this class.
  {In particular, the same stability inequality implies that, whenever the initial datum $u_0\in \ell^1(V)$, the corresponding minimal WDS coincides with the (unique) mild solution.}
\end{rem}

\section{Smoothing estimates of WDS on trees}\label{trees}
In this section we assume that the graph is a tree  with standard weights such that  {\eqref{degq+1} holds} for some integer $q\geq 2$. In Subsection \ref{Greentrees} we recalled some useful pointwise and integral estimates for the fractional Green function that will be applied below. 

\subsection{ Smoothing estimates for WDS with initial datum in $\ell^1$ on trees.}
We begin with the case of nonnegative initial data in $\ell^1(V)$.  

\begin{thm}\label{thm.smoothing.HN-like}
 Let $u$ be a WDS to \eqref{NFDE} corresponding to any nonnegative initial datum $u_0 \in \ell^1(\hn)  $. Then there exists $ C = C (s,m)>0$ such that
\begin{equation}\label{thm.smoothing.HN-like.estimate.2}
\left\| u(t) \right\|_{\ell^\infty(\hn)}  \leq \frac{C}{t^{\frac{1}{m-1}}} \left[1+ \log\!\left(t^{\frac{1}{m-1}} \left\| u_0 \right\|_{\ell^1(\hn)} \right)\right]^{\frac{s}{m-1}} \qquad \forall t\geq \|u_0\|_{\ell^1(\hn)}^{-(m-1)}\,.
\end{equation}

\end{thm}

\begin{rem}\label{tempi-piccoli}
The logarithmic correction in \eqref{thm.smoothing.HN-like.estimate.2} is analogous to the sharp smoothing effects established for the porous medium equation ($s=1$) on the hyperbolic space and, more generally, on Cartan-Hadamard manifolds with spectral gap; see \cite{GM,V4}. Such a similarity is consistent with the well-known analogy between trees and negatively curved geometries.
If the WDS considered in Theorem \ref{thm.smoothing.HN-like} is the {minimal} WDS solution corresponding to $u_0 \in \ell^1(\hn) $ (see Remark \ref{minimal WDS}), by Proposition \ref{BCPP-Thm} we have that $$\left\| u(t) \right\|_{\ell^\infty(\hn)}  \leq  {\frac{1}{d_0}}\left\| u(t) \right\|_{\ell^1(\hn)}  \leq {\frac{1}{d_0}}\|u_0\|_{\ell^1(\hn)} \qquad \forall t\geq 0\,.$$
  {In this case, the estimate above complements \eqref{thm.smoothing.HN-like.estimate.2} by providing an $\ell^\infty$ bound for small times.}

\end{rem}

\begin{proof}

 Let $u$ be a WDS to \eqref{NFDE} corresponding to any nonnegative initial datum $u_0 \in \ell^1(\hn)  $. For $ t_0>0 $ fixed, writing \eqref{main-estimate-1} for $t_1 = 2t_0$, we get:
	\begin{align}\label{thm.smoothing.HN-like.proof.1}
	u^m&(t_0, x_0) 
	\leq \frac{2^{\frac{m}{m-1}}}{ t_0} \sum_{x \in V} u(t_0, x) \, \GM^s(x, x_0) \, \dg(x)=(I)+(II)\\
	&= \frac{ 2^{\frac{m}{m-1}}}{ t_0} \sum_{x \in B_{R}(x_0)}   u(t_0, x) \, \GM^s(x, x_0)  \, \dg(x)
	+  \frac{  2^{\frac{m}{m-1}}}{ t_0} \sum_{x \in \hn\setminus B_{R}(x_0)}  \, u(t_0, x)  \, \GM^s(x, x_0) \, \dg(x) \, , \nonumber
	\end{align}
	for all $R\in \N$. 
	
	In order to estimate $ (I) $, we first notice that, by combining \eqref{comp:green} and \eqref{stimaGs}, we get
	\begin{align*}
	\notag\sum_{x \in B_{R}(x_0)}    \GM^s(x, x_0)  \, \dg(x) &\le   C \frac{q}{s}\sum_{j=0}^R q^jq^{-j}(1+j)^{s-1}\le C (1+R)^{s}  \qquad \text{for all } R\in \N \, .
	\end{align*}
    where $C>0$ depends on $s$ and $q$. Then, by Young's inequality, it follows that
	\begin{align}\label{thm.smoothing.HN-like.proof.2}
	(I)&\le \frac{1}{m} \left\| u(t_0) \right\|_{\ell^\infty(\hn)}^m + \frac{C}{t_0^{\frac{m}{m-1}}} \left( \sum_{x \in B_{R}(x_0)} \GM^s(x, x_0) \, \dg(x) \right)^{\frac{m}{m-1}} \\
	&\le \frac{1}{m} \left\|u(t_0)\right\|_{\ell^\infty(\hn)}^m+\frac{C}{t_0^{\frac{m}{m-1}}}	\, (1+R)^{\frac{sm}{m-1}} \qquad \text{for all } R\in \N \,,\nonumber
	\end{align}
	where $C>0$ is another constant depending only on $q,s,m$.

To estimate  $ (II) $ we observe that, for all $R\in \N$, by \eqref{stimaGs} it holds
\begin{align}\label{thm.smoothing.HN-like.proof.3}
(II) &\le \frac{C}{t_0} \sum_{x \in \hn\setminus B_{R}(x_0)}  \, u(t_0, x)  \,  (d(x,x_0)+1)^{s-1} q^{-d(x,x_0)} \, \dg(x) \\
& \leq \frac{C}{t_0} \, \frac{\left\| u(t_0) \right\|_{\ell^1(\hn)}}{(1+R)^{1-s}q^R}  \leq \frac{C}{t_0} \, \frac{\left\| u_0 \right\|_{\ell^1(\hn)}}{(1+R)^{1-s}q^R} \, . \notag
\end{align}
 Combining \eqref{thm.smoothing.HN-like.proof.1}, \eqref{thm.smoothing.HN-like.proof.3} and \eqref{thm.smoothing.HN-like.proof.2}, and taking the supremum over $x_0\in \hn$, we end up with
\begin{align*} 
\left\| u(t_0) \right\|_{\ell^\infty(\hn)}^m &\le C\,\frac{(1+R)^{\frac{sm}{m-1}} }{t_0^{\frac{m}{m-1}}} \left( 1+ \frac{t_0^{\frac{1}{m-1}}\|u_0\|_{\ell^1(\hn)}}{(1+R)^{\frac{m-1+s}{m-1}} q^R}\right)\\& \le C\,\frac{(1+R)^{\frac{sm}{m-1}} }{t_0^{\frac{m}{m-1}}} \left( 1+ \frac{t_0^{\frac{1}{m-1}}\|u_0\|_{\ell^1(\hn)}}{ q^R}\right)
 \qquad \forall R \in \N  \,  . \notag
\end{align*}
We now choose $R=[A]\in \N$ (where $[\cdot]$ denotes the integer part) with $A\geq 0$ defined by
$$  t_0^{\frac{1}{m-1}}\|u_0\|_{\ell^1(\hn)}=q^A\,.$$
This choice is admissible provided $t_0\geq \|u_0\|_{\ell^1(\hn)}^{-(m-1)}$ and allows to conclude that
\begin{align*} \label{thm.smoothing.HN-like.proof.10}
\left\| u(t_0) \right\|_{\ell^\infty(\hn)}^m &\le C\,\frac{(1+A)^{\frac{sm}{m-1}} }{t_0^{\frac{m}{m-1}}} \left( 1+ \frac{q t_0^{\frac{1}{m-1}}\|u_0\|_{\ell^1(\hn)}}{ q^A}\right)\\& 
= \frac{C (1+q)}{t_0^{\frac{m}{m-1}}} \left[1+ \log\!\left(t_0^{\frac{1}{m-1}} \left\| u_0 \right\|_{\ell^1(\hn)} \right)\right]^{\frac{sm}{m-1}}\,.
 \notag
\end{align*}
Hence, inequality \eqref{thm.smoothing.HN-like.estimate.2} follows. \end{proof}

\subsection{Smoothing estimates for WDS with initial datum in $\ell^1_{\GM^s}$ on trees}

We provide a smoothing estimate for the WDS constructed in Theorem \ref{existenceWDS}. Since it is obtained as the monotone limit of mild solutions with initial data in $\ell^1(V)$ (which are WDS in view of Proposition \ref{thm-weak-mild}), we start by showing a smoothing estimate holding for the latter.

\begin{lem}
 Let $u$ be a WDS to \eqref{NFDE} corresponding to a nonnegative initial datum $u_0 \in \ell^1(\hn)  $. Then there exists $ C = C (q,s,m)>0$ such that
\begin{equation}\label{thm.smoothing.HN-like.estimate.22}
 \left\| u(t) \right\|_{\ell^\infty(V)}  \le \dfrac{C}{t^{\frac1m}}  \left\| u_{0}\right\|_{\ell^1_{\GM^s}(V)}^{\frac1m} \qquad \forall t\geq \left\| u_{0}\right\|_{\ell^1_{\GM^s}(V)}^{-(m-1)}\,.
\end{equation}

\end{lem}

\begin{proof}
By repeating the first part of the proof of Theorem \ref{thm.smoothing.HN-like} but replacing the estimate  $ (II) $ with the following
\begin{align*}
(II) \le \frac{C}{t_0} \left\| u_{0}\right\|_{\ell^1_{x_0,\GM^s}(V)} 
\end{align*}
 and taking the supremum over $x_0\in \hn$, we end up with
\begin{align} \label{thm.smoothing.HN-like.proof.91}
\left\| u(t_0) \right\|_{\ell^\infty(\hn)}^m &\le C\,\left(\frac{(1+R)^{\frac{sm}{m-1}} }{t_0^{\frac{m}{m-1}}} + \frac{\left\| u_{0}\right\|_{\ell^1_{\GM^s}(V)}}{t_0}\right)
 \qquad \forall R \in \N  \,  ,
\end{align}
We now choose $R=[B]\in \N$ with $B\geq 0$ defined by
$$(1+B)^{\frac{sm}{m-1}}= t_0^{\frac{1}{m-1}}\|u_0\|_{\ell^1_{\GM^s}(V)}\,.$$
This choice is admissible provided $ t_0\geq \left\| u_{0}\right\|_{\ell^1_{\GM^s}(V)}^{-(m-1)}$ and inserted in \eqref{thm.smoothing.HN-like.proof.91} yields
$$\left\| u(t_0) \right\|_{\ell^\infty(\hn)}^m \le C\,\left(\frac{(1+B)^{\frac{sm}{m-1}} }{t_0^{\frac{m}{m-1}}} + \frac{\left\| u_{0}\right\|_{\ell^1_{\GM^s}(V)}}{t_0}\right)=2 C\,\frac{\left\| u_{0}\right\|_{\ell^1_{\GM^s}(V)}}{t_0}$$
which yields the desired conclusion. \end{proof}

Finally we state our smoothing estimates for WDS with initial datum in $\ell^1_{\GM^s}$.
\begin{thm}\label{thm.smoothing.pesato}
 Let $u$ be the WDS to \eqref{NFDE}, constructed in Theorem \ref{existenceWDS}, corresponding to any nonnegative initial datum $u_{0}\in \ell^1_{\GM^s}(V)$. 
 Then there exists $ C = C (q,s,m)>0$ such that
\begin{equation*}
 \left\| u(t) \right\|_{\ell^\infty(V)}  \le \dfrac{C}{t^{\frac1m}}  \left\| u_{0}\right\|_{\ell^1_{\GM^s}(V)}^{\frac1m} \qquad \forall t\geq \left\| u_{0}\right\|_{\ell^1_{\GM^s}(V)}^{-(m-1)}\,.
\end{equation*}
\end{thm}
\begin{proof}

For a general initial datum $u_{0} \in \ell^1_{\GM^s}(V)$ we consider the corresponding approximating sequences of initial data $\{u_{0,n}\}_n \subset \ell^1(\hn)$ and corresponding WDS $\{u_n\}_n\subset  \ell^1(\hn)$ defined in the proof of Theorem \ref{existenceWDS}. From \eqref{thm.smoothing.HN-like.estimate.22} we get
 
\begin{align*}
\left\| u_n( t_0) \right\|_{\ell^\infty(V)} \leq \dfrac{1}{t_0^{\frac1m}}  \left\| u_{0,n}\right\|_{\ell^1_{\GM^s}(V)}^{\frac1m} \leq \dfrac{1}{t_0^{\frac1m}}  \left\| u_{0}\right\|_{\ell^1_{\GM^s}(V)}^{\frac1m} \qquad \forall t_0\geq \left\| u_{0}\right\|_{\ell^1_{\GM^s}(V)}^{-(m-1)}\,.
\end{align*}
since $ u_{0,n} \le u_0 $. Therefore, 
$$
 \left\| u(t_0) \right\|_{\ell^\infty(V)} \le \liminf_{n\to\infty} \left\| u_n(t_0) \right\|_{\ell^\infty(V)} \le \dfrac{1}{t_0^{\frac1m}}  \left\| u_{0}\right\|_{\ell^1_{\GM^s}(V)}^{\frac1m} \qquad \forall t_0\geq \left\| u_{0}\right\|_{\ell^1_{\GM^s}(V)}^{-(m-1)}\,.
 $$
 \end{proof}

\section{  {Comparison results and heat kernel estimates on trees}}\label{last}

In this section we prove several results on trees  {with standard weights}, including comparison arguments and heat kernel estimates, which are used at the end of the section to prove the integral estimates for the fractional Green function stated in Proposition \ref{intestimateGreentrees}. Related heat kernel comparison results on model trees were previously obtained in \cite{Woj}.

Given a radial function $u$ with respect to a point $o$ on a tree, we sometimes write $u(n)$ instead of $u(x)$, where $x$ is a vertex at distance $n$ from $o$.

Using the notation introduced in Section \ref{Greentrees}, given a radial function $u$ on $V_q$, we call $u(d(\cdot,o))$ the corresponding transplanted function on $V$ with respect to the point $o \in V$. We say that a radial function $u$ is 2-step decreasing if $$u(n+2)\le u(n) \qquad\forall n \in \mathbb N.$$

At first we prove a comparison between the Laplacian on a tree  {satisfying \eqref{degq+1}} and the Laplacian on the homogeneous tree when acting on radial 2-step decreasing functions.

\begin{lem}\label{lap_rad}
Suppose that $u$ is a radial 2-step decreasing function on $V_q$  {and that $T$ is a tree with standard weights satisfying \eqref{degq+1}}. Then, the transplanted function $u:V \to \mathbb R$ satisfies
\begin{align*}
   \Delta u(x) \le   \Delta^{(q)}u(x_q) \qquad \forall x_q \in V_q \ : d_q(x_q,o_q)=d(x,o).
\end{align*}
\end{lem}
\begin{proof}
Set $|x|=d(x,o)$. Suppose $x \not= o$. Then,
\begin{align*}
    \Delta u(x)&=-u(x)+\frac{1}{\mathrm{deg}(x)}\sum_{y \sim x}u(y)\\ &=-u(x)+\frac{1}{\mathrm{deg}(x)}\bigg((\mathrm{deg}(x)-1)u(|x|+1)+u(|x|-1)\bigg) \\ &=
    -u(x)+u(|x|+1)+\frac{1}{\mathrm{deg}(x)}\bigg(u(|x|-1)-u(|x|+1)\bigg).
\end{align*} Now we use that \begin{align*}
    u(|x|-1)\ge u(|x|+1) \qquad \text{and} \qquad
    \frac{1}{\mathrm{deg}(x)}\le\frac{1}{q+1},
\end{align*} so that
\begin{align*}
    &-u(x)+u(|x|+1)+\frac{1}{\mathrm{deg}(x)}\bigg(u(|x|-1)-u(|x|+1)\bigg) \\ &\le -u(|x|)+u(|x|+1)+\frac{1}{q+1}\bigg(u(|x|-1)-u(|x|+1)\bigg) =\Delta^{(q)}u(x_q).
\end{align*} If $x=o$,
\begin{align*}
  \Delta u(o)=-u(0)+u(1) =\Delta^{(q)}u(o_q).
\end{align*}
This completes the proof.
\end{proof}

The next result provides a class of radial 2-step decreasing functions on $V_q$ to which Lemma \ref{lap_rad} can be applied.

\begin{lem}\label{step}
    The following functions are radial 2-step decreasing:
    \begin{enumerate}
        \item[i)] $h^{(q)}_t(\cdot,o_q)$ on $V_q$;
        \item[ii)] $\sum_{y \in B_r(o_q)}u(d_q(\cdot,y))$ if $u$ is any radial 2-step decreasing function on $V_q$.
    \end{enumerate}
    \end{lem}
    \begin{proof}
    Throughout this proof for every vertex $x \in V_q$ we denote by $|x|$ the distance from $x$ to the reference vertex $o_q$.  We write for simplicity $h_t^{(q)}(|x|)$ in place of $h_t^{(q)}(x,o_q)$.
    
    We first notice that \(h_t^{\mathbb Z}\) on \(\mathbb Z\) is 2-step decreasing since
\[
h_t^{\mathbb Z}(\cdot)=e^{-t} I_{|\cdot|}(t),
\]
where \(I_n\) denotes the modified Bessel function of the first kind of order \(n\). Moreover, it is known that
\[
I_n(t)\ge I_{n+2}(t)
\qquad \forall n\in\mathbb N,
\]
see \cite[Formula (10.04), p.~60]{Olver} or \cite[Lemma 2.4(ii)]{CMS}.
      
Then, statement $i)$ follows from the above together with \cite[Proposition 2.5 (ii)]{CMS}.  Indeed, by \cite[Proposition 2.5 ii)]{CMS} we have that
        \begin{align*}
            h^{(q)}_t(|x|)=\frac{e^{(c_q-1)t}}{t \sqrt q}q^{-|x|/2}\sum_{k \ge 0}q^{-k}(|x|+2k+1)h^{\mathbb Z}_{t c_q}(|x|+2k+1),
        \end{align*} where $c_q= \frac{2q^{1/2}}{q+1}$. It follows that
        \begin{align*}
           h^{(q)}_t(|x|+2)&=\frac{e^{(c_q-1)t}}{t \sqrt q}q^{-|x|/2-1}\sum_{k \ge 0}q^{-k}(|x|+2k+3)[h^{\mathbb Z}_{t c_q}(|x|+2k+3)] \\ 
           &=\frac{e^{(c_q-1)t}}{t \sqrt q} q^{-|x|/2}\sum_{k \ge 1}q^{-k}(|x|+2k+1)[h^{\mathbb Z}_{tc_q}(|x|+2k+1)].
        \end{align*} Hence, \begin{align*}
           &h^{(q)}_t(|x|)-h^{(q)}_t(|x|+2)=\frac{e^{(c_q-1)t}}{t \sqrt q}q^{-|x|/2}(|x|+1) h^{\mathbb Z}_{tc_q}(|x|+1)\ge 0 ,
        \end{align*}and this proves $i)$.

        Now we turn to the proof of $ii)$, see Figure \ref{figura} for a visual aid. Let $u$ be any radial 2-step decreasing function on $V_q$. Fix $o_q \in V_q$ and $r \in \mathbb N$.
Set
\[
F(x)=\sum_{y\in B_r(o_q)} u(d_q(x,y))
= \sum_{k\ge 0} u(k)\, a_k(x),\]

where $a_k(x)=|S_k(x)\cap B_r(o_q)|$ and $S_k(x)=\{y \in V_q \ : d_q(x,y)=k\}$. Notice that by homogeneity $a_k$ is radial and thus $F$ is radial.  Moreover, $a_k(|x|)=0$ if $k< |x|-r$ or $k>|x|+r$.

Let $x^{(2)}$ be a vertex at distance $2$ from $x$ along a geodesic moving away from $o_q$. Observe that $a_k(x^{(2)})=a_k(|x|+2)$.

The behaviour of the coefficients $a_k(x)$ depends on the relative position of $x$ with respect to the ball $B_r(o_q)$. Accordingly, we distinguish three regimes.

\medskip

\noindent
\textbf{Case 1: $|x|\ge r$.} In this case  $x^{(2)}\not\in B_r(o_q)$ and

\[
F(x)= \sum_{k=|x|-r}^{|x|+r} a_k(x)\, u(k), \qquad F(x^{(2)})= \sum_{k=|x|+2-r}^{|x|+2+r} a_k(x^{(2)})\, u(k).
\]

Moreover, any geodesic path from $x^{(2)}$ to a vertex in $B_r(o_q)$ must pass through $x$. 
It follows that, for all  $|x|-r\le k\le |x|+r$,
\[
a_k(x) = a_{k+2}(x^{(2)}).
\]

Therefore,
\begin{align*}
F(x)-F(x^{(2)})
&= \sum_{k=|x|-r}^{|x|+r} a_k(x)\, u(k)
 - \sum_{k=|x|+2-r}^{|x|+2+r} a_k(x^{(2)})\, u(k) = \sum_{k=|x|-r}^{|x|+r} a_k(x)\bigl(u(k)-u(k+2)\bigr)\ge 0,
\end{align*}
where we used the $2$-step monotonicity of $ {u}$. 

\medskip

\noindent
\textbf{Case 2: $|x|\le r-2$.} Notice that this assumption forces $r \ge 2$.
In this case both $x$ and $x^{(2)}$ belong to $B_r(o_q)$, hence
\[
F(x)=\sum_{k=0}^{|x|+r} a_k(x)\, u(k), 
\qquad
F(x^{(2)})=\sum_{k=0}^{|x|+2+r} a_k(x^{(2)})\, u(k).
\]

The analysis in this situation is more delicate, since the second sum contains two additional non‑zero terms, and a simple change of variables does not align the summation ranges. For this reason we compare the coefficients $a_k(x)$ and $a_k(x^{(2)})$ directly.

\begin{itemize}
\item[(a)] If $0\le k \le r-|x|-2$, then the spheres $S_k(x)$ and $S_k(x^{(2)})$ are entirely contained in $B_r(o_q)$, hence by homogeneity  
\[
a_k(x) =|S_k(x)|=|S_k(x^{(2)})|=a_k(x^{(2)}).
\]

\item[(b)] If $r-|x|-1\le k\le r-|x|$, then $S_k(x)\subset B_r(o_q)$, while $S_k(x^{(2)})$ is partially outside $B_r(o_q)$. In this case one checks that $S_k(x^{(2)})\cap B_r(o_q)=S_k(x^{(2)})\setminus E_k$ where $E_k=\{y \in S_k(x^{(2)}) \ : d(y,o_q)=|x|+2+k\}$. It follows that
\[   a_k(x)=q^{k-1}(q+1) \quad \text{ and }\quad 
a_k(x^{(2)})=q^{k-1}(q+1)-q^k=q^{k-1}.
\]
Moreover, we have that $S_{k+2}(x^{(2)})\cap B_r(o_q)=S_{k+2}(x^{(2)}) \setminus F_k$, where $F_k=\{y \in S_{k+2}(x^{(2)}) \ : d(y,o_q)\in \left\{|x|+2+k,|x|+2+k+2\}\right\}$, thus a direct computation shows that $$a_{k+2}(x^{(2)})=|S_{k+2}(x^{(2)})\cap B_r(o_q)|=q^{k+1}(q+1)-q^{k+2}-(q-1)q^{k}=q^k.$$

\item[(c)] If $r-|x|+1\le k\le r+|x|$, then any vertex in $S_{k+2}(x^{(2)})\cap B_r(o_q)$ is reached by a geodesic starting from $x^{(2)}$ and passing through $x$. Consequently,
\[
a_k(x)=a_{k+2}(x^{(2)}).
\]
\end{itemize}

It follows that
\begin{align*}
&F(x)-F(x^{(2)})=\sum_{k=0}^{r+|x|} a_k(x)\, u(k)
   - \sum_{k=0}^{r+|x|+2} a_{k}(x^{(2)})u(k)\\
   &\overset{(a)}{=}\sum_{k=r-|x|-1}^{r+|x|} a_k(x)\, u(k)
   - \sum_{k=r-|x|-1}^{r+|x|+2} a_{k}(x^{(2)})u(k)\\
   &=\sum_{k=r-|x|-1}^{r-|x|} a_k(x)\, u(k)+\sum_{k=r-|x|+1}^{r+|x|}a_k(x)u(k)-\sum_{k=r-|x|-1}^{r-|x|} a_{k}(x^{(2)})u(k)-\sum_{k=r-|x|+1}^{r+|x|+2} a_{k}(x^{(2)})u(k)\\
&\overset{(b)}{=} \sum_{k=r-|x|-1}^{r-|x|} (q^{k-1}(q+1)-q^{k-1})\, u(k)
   + \sum_{k=r-|x|+1}^{r+|x|} a_k(x) u(k)-\sum_{k=r-|x|+1}^{r+|x|+2} a_k(x^{(2)}) u(k) \\&=\sum_{k=r-|x|-1}^{r-|x|} q^k\, u(k)
   + \sum_{k=r-|x|+1}^{r+|x|} a_k(x) u(k)-\sum_{k=r-|x|-1}^{r+|x|} a_{k+2}(x^{(2)}) u(k+2)\\ &\overset{(b),(c)}{=}\sum_{k=r-|x|-1}^{r-|x|} q^k\, \left(u(k)-u(k+2)\right)+\sum_{k=r-|x|+1}^{r+|x|}a_k(x)\left(u(k)-u(k+2)\right) \ge 0
\end{align*}

because $u$ is 2-step decreasing. Hence, $F(x)\ge F(x^{(2)})$.

\medskip

\noindent
\textbf{Case 3: $|x|=r-1$.} In this borderline case, $r \ge1$, $x \in B_r(o_q)$ and $x^{(2)} \not \in B_r(o_q)$. We compute explicitly the first coefficients:
\[
a_0(x)=1,\quad a_0(x^{(2)})=0,
\qquad 
a_1(x)=q+1,\quad a_1(x^{(2)})=1,
 \qquad a_2(x^{(2)})=1, \qquad a_3(x^{(2)})=q.\]
For $k\ge2$, the same argument as in Case 1 shows that $a_k(x)=a_{k+2}(x^{(2)})$. Therefore,
\begin{align*}
F(x)-F(x^{(2)})&=\sum_{k=0}^{|x|+r}a_k(x) u(k)-\sum_{k=1}^{|x|+2+r}a_k(x^{(2)})u(k)\\&=\sum_{k=0}^1a_k(x) u(k)+\sum_{k=2}^{|x|+r}a_k(x) u(k)-\sum_{k=1}^{3}a_k(x^{(2)})u(k)-\sum_{k=4}^{|x|+2+r}a_k(x^{(2)})u(k)\\
&= u(0)+(q+1)u(1)-u(1)-u(2)-qu(3)  + \sum_{k=2}^{|x|+r} a_k(x)\bigl(u(k)-u(k+2)\bigr)\\ &=\sum_{k=0}^1q^k [u(k)-u(k+2)]+\sum_{k=2}^{|x|+r} a_k(x)\bigl(u(k)-u(k+2)\bigr)\ge 0.
\end{align*}
Thus $F(x)\ge F(x^{(2)})$ also in this case.\\

We have thus proved that $F(x)\ge F(x^{(2)})$ for every $x$, concluding the proof of $ii)$.
\end{proof}

  \begin{center}
            
        \end{center}

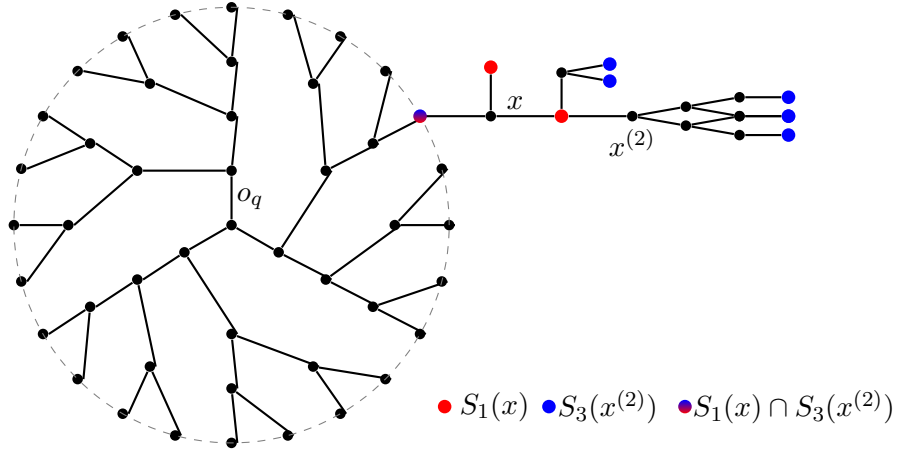
\begin{figure}
\begin{tikzpicture}[scale=0.6,
    node/.style={circle,fill=black,inner sep=1.4pt},
    rednode/.style={circle,fill=red,inner sep=1.8pt},
    bluenode/.style={circle,fill=blue,inner sep=1.8pt},
  greennode/.style={
    circle,
    inner sep=1.8pt,
    shading=axis,
    left color=blue!100,
    right color=red!100,
    shading angle=0
},
    edge/.style={thick},
    labelnode/.style={circle,fill=black,inner sep=1.4pt}
]

\def\R{1.2}

\node[labelnode,label={[xshift=7pt] {$o_q$}}] (L0-0) at (0,0) {};

\foreach \i in {0,...,2} {
  \node[node] (L1-\i) at ({90+120*\i}:{1*\R}) {};
  \draw[edge] (L0-0) -- (L1-\i);
}

\foreach \i in {0,...,5} {
  \node[node] (L2-\i) at ({90+60*\i}:{2*\R}) {};
}

\foreach \i in {0,1,2} {
  \pgfmathsetmacro{\cA}{mod(2*\i,6)}
  \pgfmathsetmacro{\cB}{mod(2*\i+1,6)}
  \draw[edge] (L1-\i) -- (L2-\cA);
  \draw[edge] (L1-\i) -- (L2-\cB);
}

\foreach \i in {0,...,11} {
  \node[node] (L3-\i) at ({90+30*\i}:{3*\R}) {};
}

\foreach \i in {0,...,5} {
  \pgfmathsetmacro{\cA}{mod(2*\i,12)}
  \pgfmathsetmacro{\cB}{mod(2*\i+1,12)}
  \draw[edge] (L2-\i) -- (L3-\cA);
  \draw[edge] (L2-\i) -- (L3-\cB);
}

\foreach \i in {0,...,23} {
  \node[node] (L4-\i) at ({90+15*\i}:{4*\R}) {};
}

\foreach \i in {0,...,11} {
  \pgfmathsetmacro{\cA}{mod(2*\i,24)}
  \pgfmathsetmacro{\cB}{mod(2*\i+1,24)}
  \draw[edge] (L3-\i) -- (L4-\cA);
  \draw[edge] (L3-\i) -- (L4-\cB);
}

\draw[dashed,gray] (0,0) circle (4*\R);


\coordinate (base) at (L4-20);

\node[labelnode,label={[yshift=5pt]right:$x$}] (x)
      at ($(base)+(1.3*\R,0)$) {};
\draw[edge] (base) -- (x);

\node[rednode] (xU) at ($(x)+(0,0.9*\R)$) {};
\draw[edge] (x) -- (xU);

\node[rednode] (x1) at ($(x)+(1.3*\R,0)$) {};
\draw[edge] (x) -- (x1);

\node[greennode] at (base) {};

\node[node] (x1U) at ($(x1)+(0,0.8*\R)$) {};
\draw[edge] (x1) -- (x1U);

\node[bluenode] (x1U1) at ($(x1U)+(10:0.9*\R)$) {};
\node[bluenode] (x1U2) at ($(x1U)+(-10:0.9*\R)$) {};
\draw[edge] (x1U) -- (x1U1);
\draw[edge] (x1U) -- (x1U2);

\node[labelnode,label={[yshift=-21pt]:$x^{(2)}$}] (x2)
      at ($(x1)+(1.3*\R,0)$) {};
\draw[edge] (x1) -- (x2);

\node[node] (x2A) at ($(x2)+(10:1.0*\R)$) {};
\node[node] (x2B) at ($(x2)+(-10:1.0*\R)$) {};
\draw[edge] (x2) -- (x2A);
\draw[edge] (x2) -- (x2B);

\node[node] (x3A1) at ($(x2A)+(10:1.0*\R)$) {};
\node[node] (x3A2) at ($(x2A)+(-10:1.0*\R)$) {};
\node[node] (x3B1) at ($(x2B)+(10:1.0*\R)$) {};
\node[node] (x3B2) at ($(x2B)+(-10:1.0*\R)$) {};

\draw[edge] (x2A) -- (x3A1);
\draw[edge] (x2A) -- (x3A2);
\draw[edge] (x2B) -- (x3B1);
\draw[edge] (x2B) -- (x3B2);

\node[bluenode] (x4A1) at ($(x3A1)+(0.9*\R,0)$) {};
\node[bluenode] (x4A2) at ($(x3A2)+(0.9*\R,0)$) {};
\node[bluenode] (x4B1) at ($(x3B1)+(0.9*\R,0)$) {};
\node[bluenode] (x4B2) at ($(x3B2)+(0.9*\R,0)$) {};

\draw[edge] (x3A1) -- (x4A1);
\draw[edge] (x3A2) -- (x4A2);
\draw[edge] (x3B1) -- (x4B1);
\draw[edge] (x3B2) -- (x4B2);


\begin{scope}[shift={(3,-4)}]
    \node[rednode]   at (1.7,0) {};
    \node            at (2.8,0) {$S_1(x)$};

    \node[bluenode]  at (4,0) {};
    \node            at (5.3,0) {$S_3(x^{(2)})$};

    \node[greennode] at (7,0) {};
    \node            at (9.4,0) {$S_1(x)\cap S_3(x^{(2)})$};
\end{scope}

\end{tikzpicture}
\caption{A visual representation of the argument used in the proof of Lemma \ref{step}.}\label{figura}
\end{figure}


\begin{rem}\label{rad2}
Since the heat kernel on \(T_q\) is radial, namely
\(h_t^{(q)}(x,y)=h_t^{(q)}(d(x,y))\) for all \(x,y\in V_q\), 
the previous theorem shows that for every $r \in \mathbb{N}$ the function
\[
x \longmapsto \sum_{y\in B_r(o_q)} h_t^{(q)}(x,y)
\]
is radial and 2-step decreasing on $V_q$.
\end{rem}

We shall also need the following maximum principle.

\begin{lem}\label{lem:punz}
    Assume that $z$ is a subsolution of the heat equation with initial datum zero, i.e.,
    $$ \begin{cases}
        \partial_t z-\Delta z\le 0 \qquad\mathrm{on} \ (0,T]\times V, \\
        z(0,\cdot)\le 0 \ \ \ \ \qquad\mathrm{on} \ V.
 \end{cases}$$
    
     Assume that 
    \begin{align*}
        \lim_{n \to \infty} \sup_{y \in B_n(o)^c} \sup_{t\in[0,T]} z(t,y) \le 0.
    \end{align*}
    
    Then, $z\le 0$ on $[0,T]\times V$.
\end{lem}
\begin{proof}
This is a particular case of \cite[Proposition 3.3]{Punzo} when $Z=1$. 
\end{proof}

The following result provides the decay at infinity of the heat kernel that will be needed to apply Lemma \ref{lem:punz} in Proposition \ref{prop:sub} below.
 
\begin{lem}\label{lem:heat}

Let \(h_t(x,y)\) be the heat kernel of the semigroup \(e^{t\Delta}\) on a tree satisfying \eqref{degq+1}.
Then, for every $o,y\in V$ and every \(T>0\),
\[
    \lim_{n\to\infty}\ \sup_{{(t,x)\in[0,T]\times B_n(o)^c}} h_t(x,y)=0.
\]
\end{lem}

 {
\begin{rem}
Inspecting the proof of Lemma \ref{lem:heat}, one sees that neither the tree structure nor the specific form of the Laplacian are used. The argument only relies on the positivity of the heat kernel, the estimate
\[
\sum_{z\in V} h_t(x,z)\mu(z)\le 1,
\]
the lower bound $\mu(x)\ge c>0$, and the boundedness of the generator of the semigroup on $\ell^\infty(\mu)$. Therefore the conclusion remains valid for a much larger class of semigroups on connected locally finite weighted graphs.
\end{rem}}
\begin{proof}
Fix $y\in V$. We say that a sequence $\{x_n\}_n\subset V$ \emph{diverges} (and write $x_n\to\infty$) if for every $r\in\mathbb{N}$ one has $x_n\in B_r(o)^c$ eventually (this is equivalent to say that $\lim_{n \to \infty }d(x_n,o)=+\infty$).

We first prove pointwise decay at fixed time. Let $t'\in[0,T]$. By \eqref{degq+1} and the semigroup is Markovian, we have
\[
    \sum_{x\in V} h_{t'}(x,y)
    \le \frac{1}{q+1}\sum_{x\in V} h_{t'}(x,y)\,\deg(x)
    \le \frac{1}{q+1}.
\]
Let $\{x_n\}$ be any diverging sequence. If $\lim_{n\to\infty} h_{t'}(x_n,y)\neq 0$, then there exist $\varepsilon>0$ and a subsequence $\{x_{n_k}\}$ such that $h_{t'}(x_{n_k},y)\ge \varepsilon$ for all $k$, contradicting the summability above. Therefore,
\begin{equation}\label{eq:spatial_decay_fixed_time}
    \lim_{n\to\infty} h_{t'}(x_n,y)=0
    \qquad \text{for every } t'\in[0,T].
\end{equation}
In particular, $\|h_{t'}(\cdot,y)\|_{\ell^\infty(V)}\le \sum_{x\in V} h_{t'}(x,y)\le 1/(q+1)$.

\medskip

Since $\partial_t h_t(\cdot,y)=\Delta h_t(\cdot,y)$, we obtain, for all $(t,x)  {\in [0,T]\times V}$,
\[
|\partial_t h_t(x,y)|
= |\Delta h_t(\cdot,y) {(x)}|
\le  \frac1{\deg(x)}\sum_{z\sim x}|h_t(z,y)-h_t(x,y)|
\le 2\,\|h_t(\cdot,y)\|_{\ell^\infty(V)}
\le \frac{2}{q+1}.
\]
Hence,
\begin{equation}\label{eq:uniform_Lipschitz}
    |h_{t}(x,y)-h_{s}(x,y)| \le \frac{2}{q+1}\,|t-s|
    \qquad \forall\, s,t\in[0,T],\ \forall\, x\in V.
\end{equation}

\medskip

For each $n$, set
\[
    M_n = \sup_{(t,x)\in[0,T]\times B_n(o)^c} h_t(x,y).
\]
By the definition of supremum, there exist $(t_m,x_m)\in[0,T]\times B_n(o)^c$ such that
\begin{equation}\label{eq:approx_sup}
    \lim_{m\to\infty} h_{t_m}(x_m,y) = M_n.
\end{equation}
Since $[0,T]$ is compact, up to a subsequence we may assume that $\displaystyle{\lim_{m\to\infty} t_m = \tau_n\in[0,T]}$.
By \eqref{eq:uniform_Lipschitz},
\begin{equation}\label{eq:time_continuity}
    \lim_{m\to\infty}\bigl|h_{t_m}(x_m,y)-h_{\tau_n}(x_m,y)\bigr|
    \le \lim_{m\to\infty} \frac{2}{q+1}\,|t_m-\tau_n| = 0.
\end{equation}

If $\{x_m\}$ diverges, then \eqref{eq:spatial_decay_fixed_time} (applied at the fixed time $\tau_n$) gives
\begin{equation}\label{eq:spatial_decay_on_sequence}
    \lim_{m\to\infty} h_{\tau_n}(x_m,y)=0,
\end{equation}
and combining \eqref{eq:time_continuity} with \eqref{eq:spatial_decay_on_sequence} yields $\displaystyle{ \lim_{m\to\infty} h_{t_m}(x_m,y)=0}$.
Together with \eqref{eq:approx_sup}, this would imply $M_n=0$.  Since $\{M_n\}_n$ is non-increasing and nonnegative, it follows that $M_k=0$ for all $k \ge n$, and hence $\lim_{n\to\infty} M_n=0$.

Conversely, if $\{x_m\}$ does not diverge, it admits a bounded subsequence which is eventually constant. Thus, up to extracting a subsequence, there exists $v_n\in B_n(o)^c$ such that $x_m=v_n$ for all large $m$, and by \eqref{eq:approx_sup}--\eqref{eq:time_continuity},
\[
    M_n=\lim_{m\to\infty} h_{t_m}(x_m,y)=h_{\tau_n}(v_n,y).
\]

By compactness of $[0,T]$, up to subsequences we may assume $\tau_n\to\tau'\in[0,T]$. Using \eqref{eq:uniform_Lipschitz},
\[
    \lim_{n\to\infty} \bigl|h_{\tau_n}(v_n,y)-h_{\tau'}(v_n,y)\bigr|
    \le \lim_{n\to\infty} \frac{2}{q+1}\,|\tau_n-\tau'| = 0.
\]
Since $v_n\in B_n(o)^c$, we have $v_n\to\infty$, hence by \eqref{eq:spatial_decay_fixed_time} (with $t'=\tau'$), $\displaystyle{\lim_{n\to\infty} h_{\tau'}(v_n,y)=0}$. Therefore,
\[
    \lim_{n\to\infty} M_n = \lim_{n\to\infty} h_{\tau_n}(v_n,y)=0,
\]
which proves the claim.
\end{proof}

We now introduce the auxiliary function $z$ whose properties will eventually yield the comparison estimate for heat kernel averages over balls stated in Proposition \ref{intestimateGreentrees}.

\begin{prop}\label{prop:sub}
Fix $r \in \mathbb N$, $o \in V$ and $o_q \in V_q$. For all $(t,x)\in (0,T)\times V$, define $$z(t,x)=\sum_{y \in B_r(o)} h_t(x,y) \mathrm{deg}(y)-\sum_{y \in B_r(o_q)}h_t^{(q)}(x_q,y)(q+1),$$ where $x_q$ is any vertex in $V_q$ such that $d_q(x_q,o_q)=d(x,o)$. Then, $z$ is a subsolution of the heat equation with zero initial datum and
\begin{align}\label{hp1prop}
    \lim_{n \to \infty} \sup_{x \in B_n(o)^c} \sup_{t \in [0,T]} z(t,x) \le 0.
\end{align}
\end{prop}
\begin{proof}

 Consider the solutions of the Cauchy problems
\[
\begin{cases}
\partial_t u-\Delta^{(q)}u=0 & \text{on } (0,T)\times V_q,\\
u_0=\chi_{B_r(o_q)}
\end{cases}
\qquad
\begin{cases}
\partial_t v-\Delta v=0 & \text{on } (0,T)\times V,\\
v_0=\chi_{B_r(o)}.
\end{cases}
\]
They are given by
\[
u(t,x)=(q+1)\sum_{y\in B_r(o_q)} h_t^{(q)}(x,y),
\qquad \text{and} \qquad
v(t,x)=\sum_{y\in B_r(o)} h_t(x,y)\deg(y).
\]
 Given a vertex $x \in V$ we denote by $x_q$ any vertex in $V_q$ such that $d(x,o)=d_q(x_q,o_q)$.
 
Define
\[
V\ni x \mapsto  w(t,x)=\sum_{y \in B_r(o_q)}h_t^{(q)}(x_q,y) (q+1)
\]
We show that $w$ is a supersolution of the Cauchy problem on $V$.
By Remark \ref{rad2}, the function
\(w(t,\cdot)\) is radial and \(2\)-step decreasing on \(V_q\). Hence,
Lemma \ref{lap_rad} yields
$
\Delta w(t,\cdot)(x)\le \Delta^{(q)} w(t,\cdot)(x_q)
$
and, in turn, for all $ (t,x)\in (0,T)\times V$, we get
\begin{align*}
    \partial_t w(t,x)&=\sum_{y \in B_r(o_q)}\partial_th^{(q)}_t(x_q,y)(q+1)=\sum_{y \in B_r(o_q)}\Delta^{(q)}h^{(q)}_t(\cdot,y)(x_q)(q+1)\\&=\Delta^{(q)}\bigg(\sum_{y \in B_r(o_q)}h^{(q)}_t(\cdot,y)(q+1)\bigg)(x_q)\ge \Delta\bigg(\sum_{y \in B_r(o_q)}h^{(q)}_t(\cdot,y)(q+1)\bigg)(x)=\Delta w(t,x).
\end{align*}
In particular, for all $ (t,x)\in (0,T)\times V$, $z=v-w$ satisfies:
\[
z(0,x)=\chi_{B_r(o)}(x)-\chi_{B_r(o_q)}(x_q)=0,
\qquad
\partial_t z=\partial_t v-\partial_t w\le \Delta z.
\]

   It remains to prove \eqref{hp1prop}. Observe that
    \begin{align*}
        z(t,x) \le 2\mathrm{deg}(B_r(o))z'(t,x)\,.
    \end{align*}
     Then \eqref{hp1prop} follows once proved that $z'(t,x)=\sup_{y \in B_r(o)}h_t(x,y)+\sup_{y \in B_r(o_q)}h_t^{(q)}(x_q,y)$ satisfies \eqref{hp1prop}. Since the suprema appearing in the definition of $z'$ are in fact  maxima, this follows by Lemma \ref{lem:heat} which implies that for each fixed $y \in B_r(o)$ and $y_q \in B_r(o_q)$ 
    \begin{align*}
      \lim_{n \to \infty} \sup_{x \in B_n(o)^c}\sup_{t \in [0,T]} h_{t}(x,y)=0\quad \text{and}\quad \lim_{n \to \infty} \sup_{x \in B_n(o_q)^c}\sup_{t \in [0,T]} h_{t}^{(q)}(x,y)=0\,.
    \end{align*}  
\end{proof}

We can now give the proof of Proposition \ref{intestimateGreentrees} by passing from the heat kernel comparison estimates to the corresponding estimates for the fractional Green function.

{\bf Proof of Proposition \ref{intestimateGreentrees}.}

   Inequality \eqref{comp:heat} directly follows by Proposition \ref{prop:sub} and Lemma \ref{lem:punz}. To prove \eqref{comp:green} it suffices to observe that
\begin{align*}
    \sum_{y\in B_r(o)}\GM^s(x,y) \mathrm{deg}(y)&=\int_{0}^\infty \sum_{y\in B_r(o)}h_t(x,y)\mathrm{deg}(y)\frac{dt}{t^{1-s}} \\&\le \int_{0}^\infty \sum_{y\in B_r(o_q)}h_t^{(q)}(x_q,y)(q+1)\frac{dt}{t^{1-s}}=\sum_{y\in B_r(o_q)}\GM^s_q(x_q,y)(q+1).
\end{align*}
\qed

\par\bigskip\noindent
\textbf{Acknowledgments. }The authors  are members of the Gruppo Nazionale per l'Analisi Matematica, la Probabilit\`a e le loro Applicazioni (GNAMPA) of the Istituto Nazionale di Alta Matematica (INdAM).


\end{document}